\documentclass[11pt]{amsart}

\usepackage[dvips,top=3.5cm,left=3.8cm,right=3.6cm,
foot=3.5cm,bottom=4.0cm]{geometry}
\usepackage{amsfonts,amssymb,amsmath,amsthm, latexsym}
\usepackage{url}
\usepackage{enumerate}
\usepackage{graphicx}
\usepackage{color}

  \usepackage{color}
\usepackage{amsthm,amssymb,amsmath,latexsym}

\newtheorem{theorem}{Theorem}[section]
\newtheorem{lemma}[theorem]{Lemma}
\newtheorem{proposition}[theorem]{Proposition}
\newtheorem{corollary}[theorem]{Corollary}
\theoremstyle{definition}
\newtheorem{definition}[theorem]{Definition}
\newtheorem{example}[theorem]{Example}
\theoremstyle{remark}
\newtheorem{remark}[theorem]{Remark}

\numberwithin{equation}{section}

\numberwithin{equation}{section}
\newcommand{\be}{\begin{equation}}
\newcommand{\ee}{\end{equation}}

\newcommand{\N}{{\mathbb N}}
\newcommand{\Z}{{\mathbb Z}}
\newcommand{\R}{{\mathbb R}}

\newcommand{\h}{{\mathcal H}}
\def\ol{\overline}
\def\wt{\widetilde}

\begin{document}

\title{Monopoles, dipoles, and harmonic functions on Bratteli diagrams}

\author{Sergey Bezuglyi}
\address{Department of Mathematics, Institute for Low Temperature Physics, Kharkiv 61103, Ukraine}
\curraddr{Department of Mathematics, University of Iowa, Iowa City,
52242 IA, USA}
\email{bezuglyi@gmail.com}

\author{Palle E.T. Jorgensen}
\address{Department of Mathematics, University of Iowa, Iowa City,
52242 IA, USA}
\email{palle-jorgensen@uiowa.edu}


\dedicatory{To the memory of Ola Bratteli}


\begin{abstract}

In our study of electrical networks we develop two themes: finding explicit formulas for special classes of functions  defined on the vertices of a transient network, namely monopoles, dipoles, and harmonic functions. Secondly, our interest is focused on the properties of electrical networks supported on Bratteli diagrams. We show that the structure of Bratteli diagrams allows one to describe algorithmically harmonic functions as well as monopoles and dipoles. We  also  discuss some special classes of Bratteli diagrams (stationary, Pascal, trees), and we give conditions under which the harmonic functions defined on these diagrams have finite energy.
\end{abstract}

\maketitle

\tableofcontents

\section{Introduction}

The interest in discrete harmonic analysis includes both classical roots, as well as new and recent  research directions.
But for both the classical and more recent vintage papers, the question of explicit formulas for the harmonic functions, and their finite energy properties, seems to not yet have received  systematic attention.
While there are many approaches, our present paper focuses on one in particular. To make the results more relevant, we further narrow our focus to a particular class of discrete structures, infinite sets of vertices $V$ and edges $E$.

One of the classical directions in the subject is motivated by the problem of current flow in electrical network; addressing questions like computation of electrical resistance distance, and of voltage drop between distant vertices,  induced by assignment of currents into the network in question. In these models, a graph $G := (V, E)$ (see Definition \ref{def electrical network}) represents resistors assigned to the edges of $G$, i.e., points in $E$. Since conductance is the reciprocal of electrical resistance, by an {\em electrical network}, we mean a graph $G$, and a given symmetric function $c$  defined on $E$. From this we then define a Laplace operator, a Markov model, and three Hilbert spaces, $l^2(V)$, finite energy functions on $V$, and finite dissipation functions on $E$. This is done with the two laws of electrical networks in mind, Ohm's Law, and Kirchhoff's Law; for details see for example \cite{Bott49, Chung2010, Diekman2010, Dutkay_Jorgensen2010, Dutkay_Jorgensen2011, Dutkay_Jorgensen2011-1, JorgensenTian2015, Jorgensen_Pearse2010, Jorgensen_Pearse2011, Jorgensen_Pearse2013, Jorgensen_Pearse2014, Powers1976, QuianZhang2011, Tsuchiay1967}.
    For more recent developments motivated by geometry of probability measures, infinite path space measures, sampling, boundary representations of harmonic functions, expanders, long-range order, and phase transition questions, see for example \cite{Smale_Zhou2009, SmaleZhou2009-1, Grimmett_Holvord_Peres2014}.
In addition to the themes mentioned above, we add that the past decades have seen a number of independent breakthroughs on related topics covering analysis on infinite models. Of special note and relevance are the following \cite{Ancona_Lyons_Peres1999,  Furstenberg_Weiss2003, Cartier1973, Furstenberg_Katznelson_Weiss1990, Georgakopoulus_Haeseler_Keller2015, Keller_Lenz_Warzel2013, Keller_Lenz_Schmidt_Wirth2015, Mokobodzki_Pinchon1984, Peres_Sousi2012, Petit2012}.

In order to be more specific in explaining our work, we fix the settings. They are an electrical network $(G,c)$, supported by a locally finite connected graph $G = (V,E)$ together with a symmetric conductance function $c = c_{xy}$; the Laplace operator $(\Delta u)(x) = \sum_{y \sim x} c_{xy}(u(x) - u(y))$ where $u : V \to \R$ (see Definition \ref{def of Laplace operator and harmonic}); the transition probabilities matrix $P = (p(x,y))$ whose entries determine a reversible random walk $(X_n)$ on the vertex set $V$ of $G$; the Hilbert space $\h_E$ of functions on  $V$ of finite energy (see (\ref{eq def of norm energy})). Our  interest is focused on the following important classes of functions from $\h_E$: {\em monopoles, dipoles, and harmonic functions} (see Definition \ref{def of Laplace operator and harmonic} for harmonic functions, and relations   (\ref{eq-def for w_x})  and (\ref{eq dipole v_xy}) for the definition of monopoles and dipoles). The existence and properties of these functions are closely related to various properties of electrical networks, first of all, {\em recurrence} and {\em transience} (see Definition \ref{def recurrence and transience}).

In this paper, we study  the following problems: It is well known, after the article \cite{Nash-Williams1959}, that an electrical network is transient if and only it has a monopole of finite energy. On the other hand, any network always has dipoles in the space $\h_E$. But, as far as we know, explicit formulas for monopoles and dipoles have not so far  been given for arbitrary transient networks. One of our main results is the following theorem.

\begin{theorem}
Let $(V,E,c)$ be a transient electrical network with transition probabilities matrix $P$. Let $G(x,y)$ be the Green's function determined by $P$. Then, for any vertex $x \in V$, the function
$$
w_x : a\mapsto w_x(a) = \frac{G(a,x)}{c(x)}, \ \ \ a\in V,
$$
is a monopole in $\h_E$, and
$$
a\mapsto v_{x_1,x_2}(a) := \frac{G(a,x_1)}{c(x_1)} - \frac{G(a,x_2)}{c(x_2)}, \ \ \ a\in V,
$$
defines a dipole from $\h_E$ where $x_1, x_2$ are any vertices from $V$.
\end{theorem}

The other main theme in our work is based on the notion of Bratteli diagrams (see Definition \ref{def BD}).
Bratteli diagrams were introduced in \cite{Bratteli1972}
originally as a computational device in the study of classification problems in representation theory; more specifically, the problem considered by Bratteli in 1972 was that of classifying the isomorphism classes for a certain family of $C^*$-algebras, the approximately finite-dimensional $C^*$-algebras, now referred to as AF-algebras. This problem in turn was motivated by questions in quantum statistical mechanics. But since then, these diagrams (now called Bratteli diagrams) have found numerous applications in a host of other areas of mathematics, including, combinatorics, the study of algorithms (for example, algorithms for fast discrete Fourier transforms), representation theory \cite{Bratteli_Jorgensen_Kim_Roush2002},  symbolic dynamics, including the theory of automata (in the sense of von Neumann), and the study of orbit equivalence. We mention here a few papers related to Cantor dynamics only, see e.g.  \cite{Furstenberg_Katznelson_Weiss1990, Herman_Putnam_Skau1992, Giordano_Putnam_Skau1995, BKMS2010, Durand_survey2010, Bezuglyi_Handelman2014, Bezuglyi_Karpel2015}.

 Bratteli diagrams constitute a particular class of {\em infinite graphs} (i.e., vertices and edges), see Definition \ref{def BD} below. In their original formulation by Bratteli, they were intended as a device for keeping track of multiplicity "lines" in infinite systems of finite-dimensional representations arising in inductive limit constructions. For us, the most important particularity of a Bratteli diagram $B = (V,E)$ is the fact that the vertex set is graded in levels, $V = \bigcup_{n\geq 0}V_n$,  and the edges link vertices if they are in neighboring levels only. Moreover, there are no edges between the vertices of the same level.

This structure of a Bratteli diagram makes clearer the action of the Laplace operator $\Delta$, and that of the matrix $P$. The latter is naturally partitioned into a sequence of matrices (denoted by $(\overleftarrow{P}_n)$ in the paper) such that $\overleftarrow{P}_n : \R^{|V_{n+1}|} \to  \R^{|V_{n}|}$. Every function $f$ on $V$ is identified with a sequence of vectors $(f_n)$ where $f_n \in  \R^{|V_{n}|}$ is the restriction of $f$ to $V_n$, see Figure 3. Analyzing the action of $\Delta$ on $f$, we can find out conditions when $f$ is harmonic, i.e., $\Delta f =0$.

\begin{theorem}
Let $(B(V,E), c)$ be a weighted Bratteli diagram with associated sequences of matrices  $(\overleftarrow{P}_n)$. Then a sequence of vectors $f_n \in \R^{|V_n|}$ represents a harmonic function $f  = (f_n) : V \to \R$ if and only if for any $n\geq 1$
$$
f_n - \overrightarrow{P}_{n-1} f_{n-1} = \overleftarrow{P}_n f_{n+1},
$$
where $\overrightarrow{P}_{n} = \overleftarrow{P}_n^T$.
\end{theorem}

This theorem allows one to construct an algorithm for verifying that a given function is harmonic. Analogously, we can work with dipoles and monopoles determined on Bratteli diagrams. It is interesting  that one can then easily construct examples of Bratteli diagrams for which the space of harmonic functions is either trivial, or finite-dimensional, or infinite-dimensional.

The next theme  in the paper is related to an integral representation of harmonic functions defined on a transient Bratteli diagram network via the Poisson kernel. The following result is in a spirit of \cite[Theorem 1.1]{Ancona_Lyons_Peres1999}. We prove that if $f = (f_n), f_n = f|_{V_n},$ is a given function on $V$, then
$$
h_n(x) := \mathbb E_x( f_n \circ X_{\tau(V_n)}) = \int_{\Omega_x} f_n(X_{\tau(V_n)}(\omega)) d\mathbb P_x(\omega), \ \ n\in \N,
$$
is harmonic on $V \setminus V_n$, and $h_n(x) = f_n(x), x \in V_n$,
where $(\Omega_x, \mathbb P_x)$ is the probability space of infinite paths starting at $x$. This result is a basis for the following theorem.

\begin{theorem}
Let $f = (f_n) \geq 0$ be a function on $V$ such that $\overleftarrow{P}_n f_{n+1} = f_n$. Then the sequence $(h_n(x))$ converges pointwise to a harmonic function $H(x)$. Moreover, for every $x\in V$, there exists $n(x)$ such that $h_i(x) = H(x), i \geq n(x)$. Equivalently, the sequence $(f_n\circ X_{\tau(V_n)})$ converges in $L^1(\Omega_x, \mathbb P_x)$.
\end{theorem}

An essential part of our work is devoted to various examples. We consider trees, stationary Bratteli diagrams, and the Pascal graph to illustrate the general theory.
In all these case we give explicit formulas for harmonic functions and compute their energy.

Regarding to the estimation of the energy of harmonic functions, we prove the following lower bound.

\begin{theorem}
Let $f$ be a harmonic function on a weighted Bratteli diagram $(B, c)$.
Then
$$
\sum_{n =0}^\infty \frac{I_1^2}{\beta_n |V_n|} \leq   \|f\|_{\mathcal H_E}^2
$$
where $\beta_n = \max\{c(x) : x \in V_n\}$.
\end{theorem}

The paper is organized as follows. Section \ref{sect_Basics} contains basic definition and facts about electrical networks, random walks on the graph underlying the network, the Laplace operator $\Delta$, the Hilbert space $\h_E$ of finite energy functions, harmonic functions, Bratteli diagrams, etc. The proof are usually skipped or rather sketchy because the formulated statements are already in the literature. The only exception is the subsection devoted to Bratteli diagrams where we discuss the question: when can a locally finite graph be realized as a Bratteli diagram? Section \ref{sect monopoles dipoles} deals with mostly new results about the energy space and the properties of harmonic functions, monopoles, and dipoles.
Our main results of this section is formulated in Theorem 1.1 above. In Section \ref{sect HF on BD}, we use the structure of a Bratteli diagram and matrices associated to the diagram in order to formulate an algorithmic method of constructing harmonic functions, monopoles, and dipoles. This methods answers the question about the existence of harmonic functions and gives examples when this space is trivial. In Section \ref{sect HF Poisson kernel}, we  discuss a method of representation of harmonic functions in terms of the Poisson kernel. Also we prove a number of properties of harmonic functions defined on a Bratteli diagram. Section \ref{sect HF on trees stationary BD}, contains examples of harmonic functions. We consider three cases when harmonic functions are defined on a binary tree, the Pascal graph, and a stationary Bratteli diagram. In all cases, we give explicit formulas for some classes of harmonic functions. The final section of the paper deals with the energy space. We give an lower bound for the energy of a harmonic function defined on a Bratteli diagram. This estimate is useful for determining whether a harmonic function has infinite energy. Also we give formulas for the energy of the harmonic functions found in Section  \ref{sect HF on trees stationary BD}.

\section{Basics on electrical networks, random walks,  and Bratteli diagrams}\label{sect_Basics}

\subsection{Electrical networks}
We shall need to make use of some facts on electrical networks; -- for the benefit of readers, we have included a brief fact summary of what is needed. Systematic accounts, and related, are in \cite{Bott49, Chung2010, Diekman2010, Jorgensen_Pearse2011, QuianZhang2011, Tsuchiay1967}.

By a graph $G = (V,E)$, we mean a connected undirected locally finite graph with single edges between vertices. The vertex set $V = V(G)$ is assumed to be countably infinite, and the edge set $E = E(G)$ has no loops. The notation $y \sim x$ means that $y$ is a nearest neighbor of $x$, so  $(x y) \in E$, and the set $\{ y \in V : y \sim x\}$ of all neighbors of $x$ is finite for any vertex $x$. We also write $e = (xy)$, and $s(e) = x, r(e) =y$, when the edge $e$ links $x, y \in V$.  . For any two vertices $x, y \in V$, there exists a finite path $\gamma = (x_0, x_1, ... , x_n)$ such that $x_0 = x, x_n = y$ and $(x_ix_{i+1}) \in E$ for all $i$. We will keep the assumption  that all considered graphs are  {\em infinite} without mentioning it in the statements.

\begin{definition}\label{def electrical network}  A {\em electrical network} $(G,c)$ is a weighted graph $G$ with a symmetric {\em conductance function} $c : V\times V\to [0, \infty)$, i.e., $c_{xy} = c_{yx}$ for any $(x y) \in E$. Moreover, it is required that $c_{xy} >0$ if and only if $(xy) \in E$.  This means that the conductance  function $c$ is actually defined on the edge set  $E$ of the graph $G$. The reciprocal value $r_{xy} = 1/c_{xy}$ is called the {\em resistance} of the edge $e = (xy)$. For any $x\in V$, we define the total conductance at $x$ as
$$
c(x) := \sum_{y \sim x} c_{xy}.
$$
The function $c$ is defined for every $x \in V$ since this sum is always finite.
\end{definition}

In some cases it may be useful to represent the conductance function $c$ as $c : [(V\times  V) \setminus \mbox{Diagonal}] \to [0, \infty)$ assuming that $c_{xy} = 0$ if $x,y$ are not neighbors in $G$.

Given an electrical network $(G, c) = (V,E,c)$ with a fixed conductance function $c$, the following three Hilbert spaces will be used:
$$
l^2(V) := \mbox{all\ functions\ $u$\ on\ $V$\ such\ that\ $||u||^2_{l^2} = \sum_{x\in V} |u(x)|^2$} < \infty,
$$
$$
l^2(V, c) := \mbox{all\ functions\ $u$\ on\ $V$\ such\ that\ $||u||^2_{l^2(V,c)} = \sum_{x\in V} c(x)|u(x)|^2$} < \infty,
$$
and
$$
\h_E:= \mbox{equivalence\ classes\ of\ functions\ on\ $V$\ such\ that}
$$
\be\label{eq def of norm energy}
||u||^2_{\h_E} = \frac1{2}\sum_{(xy)\in E}c_{xy}|u(x) -u(y)|^2 < \infty,
\ee
which is called the {\em finite energy space} (we say that $u_1$ and $u_2$ are equivalent if $u_1 - u_2 =\mathrm{constant}$).

\begin{definition}\label{def of Laplace operator and harmonic}
The {\em Laplacian} on $(G,c)$ is the linear operator $\Delta$ which is defined on the linear space of functions $f : V \to \R$ by the formula
\begin{equation}\label{Laplacian formula}
(\Delta f)(x) := \sum_{y \sim x} c_{xy}(f(x) - f(y)).
\end{equation}
A function $f : V \to \R$ is called {\em harmonic} on $(G,c)$ if $\Delta f(x) = 0$ for every $x\in V$.  If (\ref{Laplacian formula}) holds at each vertex of a set $W \subset V$, then we say that $f$ is harmonic on $W$.
\end{definition}

We will be studying {\em harmonic functions in exterior domains}. They are  the solutions to the following equations. Given $(V, E, c)$ as specified above, let $\Delta$ be the corresponding Laplace operator. Fix a finite subset $F \subset V$. For functions $\psi$ on $V$, we consider the following problem
\be\label{eq harm on F^c}
\Delta \psi = 0 \ \ \ \mbox{on}\ V\setminus F,
\ee
Of special interest are the classes when $|F| =1$ and $|F| = 2$.  If $F = \{x_0\}$, then the solutions $w = w_{x_0}$ to
\be\label{eq monopoles}
\Delta w_{x_0} = \delta_{x_0}
\ee
are called {\em monopoles}. If $F = \{x_1, x_2\}, x_1 \neq x_2$, then the solutions $v = v_{x_1, x_2}$ to
\be\label{eq dipoles}
\Delta v_{x_1, x_2} = \delta_{x_1} - \delta_{x_2}
\ee
are called {\em dipoles}.

We are concerned with the following two questions regarding equations (\ref{eq harm on F^c}) - (\ref{eq dipoles}):

(i) Find explicit formulas and algorithms for solutions;

(ii) When are these solutions of finite energy?

We remark, as for (\ref{eq dipoles}), one can show, with  the aid of Riesz in $\h_E$, that for every $x_1, x_2$ there is a unique solution $v = v_{x_1, x_2} \in \h_E$ such that
\be\label{eq Riesz for dipole}
\langle  v_{x_1, x_2}, f \rangle_{\h_E} = f(x_1) - f(x_2)
\ee
holds for all $f \in \h_E$. Moreover, (\ref{eq Riesz for dipole}) implies (\ref{eq dipoles}).
\\

We denote by $\mathcal Harm$ the set of classes of harmonic functions on $(G,c)$ where two harmonic functions, $f$ and $g$, are identified if $f - g = \mathrm{const}$. Clearly, every constant function is harmonic. With some abuse of notation we will identify an element $f$ of $\mathcal Harm$ with a corresponding harmonic function. This means that we can always choose a prescribed value for $f(x_0)$ at some fixed vertex $x_0$. Usually, we will require that  $f(x_0) =0$.
We say that $\mathcal Harm$ is {\em trivial} if it reduces to constant harmonic functions.

To any conductance function $c$ on a network $G$, we can associate a {\em reversible  Markov kernel} $P = (p(x,y))_{x,y \in V} $ with transition probabilities defined by  $p(x,y) = \dfrac{c_{xy}}{c(x)}$. Then $p(x, y)c(x) = p(y, x)c(y)$ for any $x,y \in V$. We say that the process is {\em reversible}. Define the probability transition operator for $ f : V \to  \R$ by setting
\be\label{eq P}
(P f)(x) = \sum_{y \sim x} p(x, y)f(y), \ \ x \in V.
\ee

In discrete harmonic analysis, two operations play a key role, the {\em Laplacian}  $\Delta$ (see (\ref{def of Laplace operator and harmonic})), and the {\em Markov operator} $P$ (see (\ref{eq P})). For many problems, one is even used in the derivation of properties of the other.
Both represent actions (operations) on appropriate spaces of functions, functions defined on the infinite set of vertices $V$. Since $V$ is infinite, we are faced with a variety of choices of infinite-dimensional function spaces. Because of spectral theory, we shall consider Hilbert spaces. But even restricting to Hilbert spaces, there are at least three natural candidates. Which one to use depends on the operator considered, and the questions asked; see (1) -- (3) below.

    We saw that both the Laplacian  $\Delta$, and the Markov operator $P$ have infinite by infinite matrix representations. These infinite by infinite matrices are special in that they have non-zero entries only in finite bands containing the matrix-diagonal (i.e., infinite banded matrices). This makes the algebraic matrix operations well defined.

   Now passing to appropriate Hilbert spaces, we note that the Laplacian  $\Delta$  will be an unbounded operator, albeit semi-bounded. By contrast we show in Lemma 2.3 that there is a weighted $l^2$-space such that the Markov operator $P$ is bounded, self-adjoint, and it has its spectrum contained in the finite interval $[- 1, 1]$. We caution, that in general this spectrum may be continuous, or have a mix of spectral types, continuous (singular or Lebesgue), and discrete.

From \cite{Jorgensen_Pearse2011, Dutkay_Jorgensen2011}, we know that:

(1) $\Delta$ is self-adjoint, generally unbounded operator with dense domain in $l^2(V)$;

(2) $\Delta$ is Hermitian, generally unbounded operator with dense domain in $\h_E$, but, in general, it is not self-adjoint;

(3) $P$ is bounded and self-adjoint in $l^2(V,c)$.

To illustrate these statements, we give a short proof of (3).

\begin{lemma}
Let $l^2(c)$ be the space of functions $u$ on $V$ satisfying the condition
$||u||^2_{l^2(c)} = \sum_{x\in V} c(x)|u(x)|^2$, and $P$ is defined by (\ref{eq P}). Then
\be\label{eq P is self-ad}
\langle u_1, Pu_2\rangle_{l^2(c)} = \langle Pu_1, u_2\rangle_{l^2(c)},\ \ u_1, u_2 \in l^2(c).
\ee
Moreover, the spectrum of $P$ is a subset of $[-1, 1]$, and $ - I \leq P \leq I$, i.e.,
\be\label{eq P inequalities}
-||u||^2_{l^2(c)} \leq \langle u, Pu \rangle_{l^2(c)} \leq  ||u||^2_{l^2(c)},\ \ u \in l^2(c).
\ee
\end{lemma}

\begin{proof}
The fact that $P$ is a self-adjoint operator follows from the relations $c(y)p(y,x) = c(x)p(x,y)$ (the Markov kernel $P$ is reversible) and (\ref{eq P}) which are applied first to the dense subset of functions with finite support (for simplicity we assume that functions on $V$ are real-valued; the case of complex-valued functions is similar).

Next, it follows, from the inequality
$$
2\sum_{x \in V} (c(x) |u(x)|^2 - \langle u, Pu\rangle_{l^2(c)}) = \sum_{x,y \in V} c_{xy} |u(x) - u(y)|^2 \geq 0,
$$
that
$$
\langle u, Pu\rangle_{l^2(c)} \leq ||u||^2_{l^2(c)}.
$$
Similarly,
$$
2\sum_{x \in V} (c(x) |u(x)|^2 + \langle u, Pu\rangle_{l^2(c)}) = \sum_{x,y \in V} c_{xy} |u(x) + u(y)|^2 \geq 0,
$$
implies
$$
\langle u, Pu\rangle_{l^2(c)} \geq  -||u||^2_{l^2(c)}.
$$
This proves (\ref{eq P inequalities}), and $\mathrm{Spec}(P) \subset [-1,1]$.
\end{proof}

{\em Question}. Since $P$ is self-adjoint, there exists the spectral measure $E_P$ on the interval $[-1, 1]$ such that
$$
P = \int_{-1}^1 \lambda dE_P(\lambda).
$$
It would be interesting to find out under what conditions on $P$, this spectral measure has a gap, that is there exists $\lambda_0 < 1$ such that $E_P(\lambda_0, 1) =0$. More generally, one can ask about the properties of spectral measures for the operator $\Delta$ considered in both Hilbert spaces $l^2$ and $\h_E$ (see the definition of $\h_E$ above).
\\

We mention here two crucial facts about harmonic functions that will be used below. The first one, the {\em maximum principle} for harmonic functions, can be stated as follows. Let $(G, c) = (V, E, c)$ be an infinite electrical network, and let $G_1$ be a connected subgraph of $G$ with vertex set $W \subset V$. Let  $\partial W :=\{ x \in V : x \sim y \ \mbox{for \ some}\ y \in W\}$ be the outer boundary of $W$. Suppose  $h : V \to \R$ is a function that is harmonic  on $W$, and the supremum of $h$ is achieved at some point from $W$. Then the maximum principle states that $h$ is constant on $\ol W = W \cup \partial W$.

The second result is the so called {\em Dirichlet problem}. Let $(G, c)$ and $W$ be as above. The Dirichlet problem consists of solving the following  boundary problem:
\begin{equation}\label{eq Dirichlet probl}
\begin{cases}  (\Delta u) (x)  = g(x)  \ &  \mbox{for\ all }\  x\in W, \\
 u(x) = f(x)  & \mbox{for\ all}  \ x\in \partial W, \\
\end{cases}
\end{equation}
where $u : V \to \R$ is an unknown function, and the functions $g : W \to \R$ and $f : \partial W\to \R$ are given. If $W$ is finite, for all functions $g, f$ as above,
the Dirichlet problem (\ref{eq Dirichlet probl}) has a unique solution.

Suppose now that $W$ is finite, and $f$ is a function defined on $\partial W$. Then we conclude that there exists a unique harmonic function $h$ on $W$ such that $h = f$ on $\partial W$.
\\


\subsection{Infinite path space}

We shall need to make use of some facts on path-space analysis;-- for the benefit of readers, we have included a brief fact summary of what is needed. Systematic accounts, and related,  are in \cite{Du2012, Grimmett_Holvord_Peres2014, Keller_Lenz_Schmidt_Wirth2015, Peres_Sousi2012,  Woess2000, Woess2009}.

 Let $\Omega \subset V^{\infty}$ be the set of all infinite sequences $\omega =(x_i)_{i\geq 0}$ where $(x_ix_{i+1}) \in E$ for all $i$. Define $X_n : \Omega \to V$ by setting $X_n(\omega) = x_n$. Let $\Omega_x := \{\omega \in \Omega : X_0 = x\}$; then $\Omega$ is the disjoint union of subsets $\Omega_x$, $x \in V$.

Functions $e \mapsto I(e)$ on $E$ represent {\em current} in electrical network models. If $u$ is a voltage function on the vertex set $V$, define
\be\label{eq current}
I(xy) = (du)(xy) = c_{xy}(u(x) - u(y)), \ \ \  \forall e = (xy) \in E.
\ee
Setting
\be\label{eq isometry}
\h_{\mathrm{Diss}} =\{ I : E \to \R : \sum_{e\in E} \frac{1}{c_e}|I(e)|^2 < \infty\},
\ee
we define the associated Hilbert space of currents of finite dissipation.
We note that
$$
||d(u)||_{\h_{\mathrm{Diss}}} = ||u||_{\h_E}, \ \ \ \forall u \in \h_E,
$$
i.e., the mapping $d$ is an isometry from $\h_E$ into $\h_{\mathrm{Diss}}$, see (\ref{eq isometry}).

\begin{lemma}
The operator $P$ in (\ref{eq P}) defines the family of Markov measures $(\mathbb P_x : x \in V)$ such that $\mathbb P_x$ is supported by the corresponding set $\Omega_x$.
\end{lemma}

\begin{proof}
First $\mathbb P_x$ is defined on cylinder sets by the formula
$$
\mathbb P_x  (X_{1} = x_1, X_2 = x_2, ..., X_n = x_n  \ |\ X_0 =x) = p(x, x_1) p(x_1, x_2) \cdots p(x_{n-1}, x_n),
$$
and then it is extended to a probability measure on the Borel $\sigma$-algebra $\mathcal B(\Omega_x)$ by Kolmogorov consistency.
Thus, the sequence of random variables $(X_n)$ defines a {\em Markov chain} on $(\Omega_x, \mathbb P_x)$ such that the following identity holds:
$$
\mathbb P_x(X_{n+1} = y \ |\ X_n =z) = p(z,y)
$$
for any $y, z \in V$. The remaining details are obvious.
\end{proof}

Let $\lambda = (\lambda_x : x \in V)$ be a positive probability vector, $\sum_{x\in V} \lambda_x =1$. Define a probability measure $\mathbb P = \sum_{x\in V} \lambda_x \mathbb P_x$ on $\Omega$. Then $\mathbb P(\Omega_x) = \lambda_x$.

We recall the following well known result:

\begin{lemma}
The measure $\mathbb P$ is a Markov measure if and only if the probability distribution  $\lambda $ satisfies the relation $\lambda P = \lambda$, or $\sum_{y \sim x}\lambda_y p(y,x) = \lambda_x$. Furthermore,
$$
\mathbb P (X_{0} = x, X_1 = x_1, ..., X_n = x_n) =  \lambda_x p(x, x_1) p(x_1, x_2) \cdots p(x_{n-1}, x_n).
$$
\end{lemma}

We remark that the equation $\lambda P = \lambda$ may not have solutions in the set of positive probability vectors $\lambda$, i.e., $\sum_{x\in V}\lambda_x = 1$ and $\lambda_x >0$.

Since $G$ is a connected graph, the Markov chain defined by $(X_n)$ is {\em irreducible}, that is, for any $x,y \in V$ there exists $n \in \N$ such that $p^{(n)}(x,y) > 0$, where $p^{(n)}(x,y)$ is  the $xy$-entry of $P^n$. It can be seen that $p^{(n)}(x,y) = \mathbb P_x(X_n = y)$.

In a slightly different terminology, it is said that the Markov kernel $P = (p(x,y))_{x,y\in V}$ determines a {\em random walk} on the weighted graph $(G, c)$. It is known that, for an irreducible matrix $P$, the random walk on the graph $G$ will  be either recurrent or transient.

\begin{definition}\label{def recurrence and transience}
One says that the random walk on $G = (V,E)$ defined by the transition matrix $P$ is {\em recurrent} if for any vertex $x \in V$ it returns to $x$ infinitely often with probability one. Otherwise, it is called {\em transient}. Equivalently, the random walk is recurrent if and only if, for all $x, y \in V$,
\be\label{eq recurrence}
\mathbb P_x (X_n = y \ \mbox{for \ infinitely\ many \ $n$}) = 1,
\ee
and it is transient if for every finite set $F \subset V$, and for all $x \in V$,
\be\label{eq_transience}
\mathbb P_x (X_n \in F \ \mbox{for \ infinitely\ many \ $n$}) = 0.
\ee
\end{definition}

With some abuse of terminology, we say also that an electrical network $(G,c)$ is {\em recurrent/transient} if the random walk $(X_n)$ defined on the vertices of $G$ by the transition probability matrix $P$ is recurrent/transient.

We collect several useful results about electrical networks in the following statement.

\begin{lemma}\label{lem several facts on P}
For $(G,c)$, $\Delta$, and $P$ as above, the following holds:

(1) $\Delta = c(I - P)$ and
$$
f \in \mathcal Harm \ \ \Longleftrightarrow \ \ Pf = f;
$$

(2) $P 1 = 1$ and $c^* P = c^*$ where $c^* $ is $c$ considered as the row vector;

(3) $P(f^2) \geq (Pf)^2$ for any function $f : V \to \R$.

\end{lemma}

\subsection{Bratteli diagrams}\label{subsect BD} In our study of harmonic functions on infinite weighted graph we will consider a special class of such graphs, namely, Bratteli diagrams. We will see that the intrinsic structure of Bratteli diagrams influences the properties of harmonic functions defined on them. We give the definition and properties of Bratteli diagrams in this subsection.

Speaking informally, we can define a Bratteli diagram $B = B(V,E)$ as a locally finite graph whose vertex set  $V$ is a disjoint union of finite subsets (levels) $(V_n : n \in \N_0)$ such that there are no edges between vertices of the same set $V_n$  (see Figure 1 as an example of a Bratteli diagram), in more detail:

\begin{figure}[htb!]\label{fig Ex BD}
\begin{center}
\includegraphics[scale=0.7]{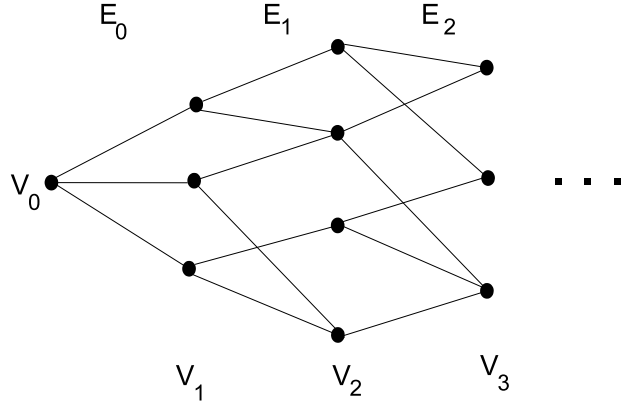}
\caption{Example of a Bratteli diagram $B =(V,E)$.}
\end{center}
\end{figure}

\begin{definition}\label{def BD}
A {\it Bratteli diagram} is a countably infinite graph $B=(V,E)$ such that the vertex
set $V =\bigcup_{i\geq 0}V_i$ and the edge set $E=\bigcup_{i\geq 0}E_i$
are partitioned into disjoint subsets $V_i$ and $E_i$ where

(i) $V_0=\{o\}$ is a single point called the top (or root) of $B$;

(ii) $V_i$ and $E_i$ are finite sets, $\forall i \geq 0$;

(iii) there exist $r : V \to E$ (range map $r$) and $s : V \to E$ (source map $s$), both from $E$ to $V$, such that $r(E_i)= V_{i+1}$, $s(E_i)= V_{i}$, and
$s^{-1}(x)\neq\emptyset$, $r^{-1}(x')\neq\emptyset$ for all $x\in V$
and $x'\in V\setminus V_0$.

The set of vertices $V_i$ is called the $i$-th level of the
diagram $B$.
\end{definition}

Given a Bratteli diagram $B$, the $n$-th {\em incidence matrix}
$A_{n}=(a^{(n)}_{x,y}),\ n\geq 0,$ is a $|V_{n}|\times |V_{n+1}|$
matrix such that $a^{(n)}_{x,y} = |\{e\in E_{n} : s(e) = x, r(e) = y\}|$
for  $x\in V_{n}$ and $y\in V_{n+1}$.

We say that a Bratteli diagram is {\em stationary} if $A_n = A$ for all $n\geq 1$.

A standard definition of a Bratteli diagram admits multiple edges between vertices of consecutive levels, i.e., $a^{(n)}_{x,y} \in \N_0$ for any vertices $x,y$ and any $n$.
On the other hand, every Bratteli diagram $B$ can be isomorphically transformed into a $0-1$ Bratteli diagram $B'$, i.e., every entry of incidence matrices of $B'$ is either 0 or 1 \cite{Herman_Putnam_Skau1992, Giordano_Putnam_Skau1995, Durand_survey2010}. Based on this observation, we will consider, without loss of generality, only 0-1 Bratteli diagrams.

\begin{definition} \label{def simple BD}
Let $B$ be  Bratteli diagram with incidence matrices $(A_n)$. Then $B$ is called {\em simple} if for any $n \geq 1$ there exists $m > n$ such that the product $A_n^m = A_n\cdots A_m >0$, i.e., all entries in $A_n^m$ are positive integers. Otherwise, $B$ is called non-simple.
\end{definition}

We note that independently of the simplicity of $B$ the graph $(V,E)$ defined by the Bratteli diagram $B$ is always connected because, for  every vertex $x\in V$, there is a finite path from the top of the diagram $o$ to $x$.

\begin{remark}[The path space $X_B$ of a Bratteli diagram $B$]:  It is customary  to assume that $X_B$ is a Cantor set, when a Bratteli diagram $B$ is considered in the context of dynamical systems. But we do not need this assumption later on, and we do not impose any restrictions to $X_B$.
\end{remark}

Let $G(V,E)$ be a connected locally finite graph. Under what conditions on $G$ can it be regarded as a Bratteli diagram? We consider here a few examples.

\begin{example}

(1) We first give an example of a  graph that cannot be represented as a Bratteli diagram; i.e., there is {\em no} system of finite sets $\{V_i\}_{i\in \N_0}$ having the properties listed in Definition \ref{def BD}. Consider a connected locally finite graph $G = (V,E)$ such that the following condition holds:
$$
\forall x\in V\ \ \exists y_1, y_2 \ \ \mbox{such\ that}\ \ y_1 \sim x, y_2\sim x \ \mbox{and} \ (y_1y_2) \in E.
$$
This means that there is no vertex in $V$ that can serve as the root of a Bratteli diagram because the set $V_1$ of the nearest neighbors always  has a  pair of vertices $y_1, y_2$ with $(y_1y_2) \in E$ .

(2) On the other hand, if $G$ is the graph, known as the   ``ladder'', then it can be represented as a Bratteli diagram, see Figure 2.

\begin{figure}[htb!]\label{fig ladder BD}
\begin{center}
\includegraphics[scale=0.8]{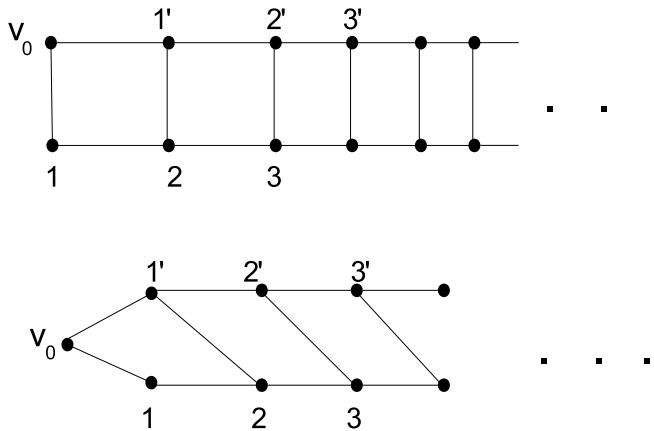}
\caption{Bratteli diagram from a ladder}
\end{center}
\end{figure}

(3) It is not difficult to give an example of a graph $G'$ that cannot be represented as a Bratteli diagram. Suppose that we start with the ``ladder'' $G$ and add new edges to $G$ by drawing the diagonals in every rectangle of Figure 2. Then we {\em claim} that $G'$ satisfies the condition given in (1) for any vertex $x \in V$. We obtain the same type of example if we start with a stationary  Bratteli diagram $B$ with incidence matrix $A =\left(
                       \begin{array}{cc}
                         1 & 1 \\
                         1 & 1 \\
                       \end{array}
                     \right)
$, and then we define a graph by  adding new edges that connect vertices $\{v_1(n), v_2(n)\} = V_n$ for every $n$. We conclude that such a graph is not a Bratteli diagram.

(4) One more example of a graph that can be viewed as a Bratteli diagram $B$ is the lattice $\Z^d$ for any $d\in \N$. To see this, we suppose that $d = 2$ for simplicity. Then we take $(0,0)$ as $V_0 = \{o\}$, and we set $V_n :=\{(x,y)\in \Z^d : |x| + |y| = n\}, n \geq 1$. Then $V_n$ is the $n$-th level of $B$. The set of edges $E_n$ between the levels $V_{n}$ and $V_{n+1}$ is inherited from the lattice. Clearly, we could take any vertex of $\Z^2$  as the root of the diagram.

\end{example}

For a finite path $\gamma(x, y)$ between $x, y \in V$, define its length $\ell(\gamma)$ as the number of edges from $E$ that form $\gamma$. Define
$$
\mathrm{dist} (x,y) = \min\{ \ell(\gamma) : \gamma(x, y)\}.
$$
Let $E(x, y)$ be the set of all finite paths $\gamma$ from $x$ to $y$.

\begin{proposition}\label{prop  graph is BD}
(1) A connected locally finite graph $G(V, E)$ has the structure of a Bratteli diagram if and only if:

(i) for every $ x \in V, \mathrm{deg}(x) \geq 2$,

(ii) there exists a vertex $x_0 \in V$ such that, for any $n \geq 1$, there are no edges between any vertices from the set $V_n :=\{y \in V : \mathrm{dist} (x_0, y) = n \}$.

(2) In general, the vertex $x_0$ is not unique: there are graphs  $G(V, E)$ that satisfy  (i) and (ii) for different vertices $x_0$ and  $y_0$ from $V$.
\end{proposition}

\begin{proof} The first part of the statement is obvious. We notice only that the requirement $\mathrm{deg}(x) \geq 2$ means that every finite path in this graph can be continued to an infinity path.

We have already mention that, for  the lattice $\Z^d$, any vertex can serve as the top (root) of a Bratteli diagram arising on $\Z^d$. We will show below that, beginning with an arbitrary diagram $B'$ with the top at $o$, and a fixed vertex $y_0\in V$, one can extend $B'$ to a new diagram $B$ by adding at most one new vertex and some edges so that both $o$ and $y_0$ will be the two roots for $B'$. To construct a graph with two different ``top'' vertices, suppose that we are given a Bratteli diagram $B' = (V', E')$ beginning at $o$. Assume, for simplicity, that  $y_0 \in V'_{1}$. Then the set $W_1$ of the nearest neighbors of $y_0$ consists of  $o$ and some vertices from $V'_2$ that form a subset $Y_2$. The other vertices from $V'_1$ are neighbors of $o$, and they are on the distance two from $y_0$. If it happens  that $Y_2 = V'_2$, then it is impossible to continue the path $(y_0, o, v_1)$ where $v_1 \in V'_1, v_1\neq y_0$. In this case we have to add a new vertex $z_1$ to the second level $V'_2$ and new edges between $V_1$ and $z_1$ to be able to construct the finite paths $(y_0, o, v_1, z_1)$ for any $v_1$. Simultaneously, we construct $W_2 = (V_1\setminus \{y_0\}) \cup Y_3$ where $Y_3$ is a subset of $V_3$ formed by the neighbors of vertices from $Y_2$. Again we repeat the described procedure if it happens that $Y_3 = V'_3$. Thus, we will produce a Bratteli diagram $B$ with two vertices  $o$ and $y_0$ serving as the roots of $B$. Regarding $o$ as the root of the diagram we obtain the new levels  $V_i \cup X_i$, and if  $y_0$ is considered as the root, then the corresponding levels are $W_i$. We notice that, by construction, there are no edges between the vertices from the same levels.  It is also clear that the same method works if one takes $y_0 \in V_m$ where $m >1$.
 \end{proof}

\begin{remark}
(1) On the other hand, it is easy to show that if we are not allowed to add new vertices, then there are Bratteli diagrams which can have only one root. To see this, take a simple Bratteli diagram $B$  such that, for any vertex $v \in V_1$ and for any vertex $w \in V_m$, there exists a finite path, $m >1$.  Then, if $w_0 \in V_m$, the vertices from $V_1$ are $(m-1)$-neighbors of $w_0$, that is $\mathrm{dist} (w_0, v) = m-1$. Hence, $o$ will be a ``sink'' for the graph whose paths start at $w_0$. This means that all paths of the form $(w_0, ... , v_1, o)$ cannot be continued.

(2) If $G = (V,E)$ admits two different representations by Bratteli diagrams, say $B = (V,E)$ and $B' = (V', E')$, then they have two different sequences of incidence matrices $(A_n)$ and $(A'_n)$ and conductance functions $c$ and $c'$. In Section \ref{sect HF on BD}, we will show how harmonic functions can be found in terms of the sequences of matrices associated to Bratteli diagrams.
\end{remark}

Let $G$ be a connected locally finite graph and $\omega \in \Omega$ be an infinite path, $\omega = (x_0, x_1, ..., )$ with $(x_ix_{i+1}) \in E$ for all $i$. We say that this path $\omega$ has no {\em self-intersections} if $x_n \notin \{x_0, x_1, ... , x_{n-1}\}$ for every $n$.

\begin{theorem}\label{thm BD in a graph}
Let $G = (V,E)$ be a connected locally finite graph that contains at least one path, $\omega$,  without self-intersection. Then $G$ contains a subgraph $H$ that is represented as a Bratteli diagram $B$ such that $\omega$ belongs to the path space $X_B$ of $B$. Moreover, the subgraph $H$ is maximal in the sense that if $H \subsetneq H'$, then $H'$ does not admit a representation as a Bratteli diagram. In particular, it can be the case  that $\omega$ is already the maximal subgraph $H$.
\end{theorem}

\begin{proof}
Let $\omega = (x_0, x_1, ... , x_n, ...)$ be an infinite path in $\Omega$ without self-intersections. We construct inductively  a Bratteli diagram $B$ whose root is $x_0$, i.e. $V_0 = \{x_0\}$, and $B$ satisfies the condition of the theorem. Define
$$
V'_1 := \{ y \in V : \mbox{dist} (x_0, y) = 1\}
$$
and
$$
V_1 := \{y \in V'_1 : (yy') \notin E, \ \ y'\in V'_1\}.
$$
Clearly, $x_1 \in V_1$. To define the next level $V_2$, we fix some $y \in V_1$ and consider $V'(y) :=\{ z\in V : z \notin V_0, \  \mbox{dist}(y, z) =1\}$. Set $V'_2 = \bigcup_{y \in V_1} V'(y)$. Then
$$
V_2 := \{z \in V'_2 : (zz') \notin E, \ \   z' \in V'_2\}.
$$
The set $V_2$ is disjoint with $V_0$ and $V_1$, and it is not empty because  $x_2\in V_2$. We see that $\mbox{dist}(x_0, x) = 2$ for every $z\in V_2$.

We now use induction. The described  procedure can be repeated word-for-word infinitely many times  for every $n$. Hence, we construct $H = \bigcup_n V_n$ as a subgraph of $G$ whose set of edges is inherited from $G$. By construction, $H$ is represented as a Bratteli diagram whose path space includes $\omega$, see Definition \ref{def BD} and Proposition \ref{prop graph is BD}.

To show that $H$ is a maximal subgraph, we suppose that $a\in V$ is such that $\mbox{dist}(x_0, a) = m$ and $a\notin V_m$. This means that if one adds $a$ to $H$, then there exists a vertex $b \in V_m$ such that $(ab) \in E$. Therefore, $H \cup \{a\}$ is not a Bratteli diagram.

 \end{proof}

\section{Energy space: monopoles and dipoles}\label{sect monopoles dipoles}

\subsection{Energy of harmonic function}
Denote by $\mathcal H_E$ the completion of functions $u : G\to \mathbb C$ with respect to the inner product
$$
\langle u, v \rangle_{\mathcal H_E} := \frac{1}{2} \sum_{x,y \in V} c_{xy}
(\ol{u(x)} - \ol{u(y)})( v(x) - v(y)).
$$
Thus,
\begin{equation}\label{norm in H_E}
\| u \|^2_{\mathcal H_E} := \frac{1}{2} \sum_{x,y \in V} c_{xy}
|u(x) - u(y)|^2.
\end{equation}
In other words, the Hilbert space $\mathcal H_E$ is formed by all functions $u$ for which the sum in (\ref{norm in H_E}) is finite. We call $\| u \|^2_{\mathcal H_E}$ the {\em energy} of the function $u$. Thus, elements of $\h_E$ are called functions of finite energy.

We remark that if $G$ is finite, then
\be\label{eq inner product finite graph}
\langle u, v \rangle_{\mathcal H_E} = \sum_{x\in G} \ol u(x)\Delta v(x)
\ee
and all harmonic functions of finite energy  are constant. Note that (\ref{eq inner product finite graph}) fails if $G$ is infinite. In the infinite case, there is a version of (\ref{eq inner product finite graph}), but it includes a second term on the right hand side that involves the boundary of $G$, see \cite{Jorgensen_Pearse2013}.
More generally, if a network $(G,c)$ is recurrent, then any harmonic function of finite energy is constant. Since $\Delta$ commutes with $f \to \ol f$, we may restrict attention to real valued functions.

The energy space $\mathcal H_E$ (see (\ref{eq def of norm energy})) for infinite graphs was  extensively studied in many papers, e.g. \cite{Jorgensen_Pearse2013, Jorgensen_Pearse2014, JorgensenTian2015}.  We mention here several notions and well known facts about the properties  of this Hilbert space.

It turns out that for harmonic functions one can find more convenient formulas for computing the energy. The following lemma is general.

\begin{lemma}\label{lem for energy of harm fns}
(i) Let $f \in \mathcal Harm_0$ on $(G,c)$. Then the energy norm can be found by the formulas:
\be\label{eq energy for harm}
\| f\|^2_{\mathcal H_E} = \frac{1}{2}\sum_{x \in V} c(x) ((Pf^2)(x) - f^2(x)),
\ee
and
\be\label{eq energy for harm 1}
\| f\|^2_{\mathcal H_E} = -\frac{1}{2} \sum_{x \in V} (\Delta f^2)(x).
\ee

(ii) If a given function $f$  on $V$ is harmonic off a finite set $F \subset V$, then it has finite energy if and only if the sums in (\ref{eq energy for harm}) and (\ref{eq energy for harm 1}) are finite.
\end{lemma}

\begin{proof}
While the proof of (i) can be found in various papers, see e.g. \cite{Jorgensen_Pearse2011}, in order to highlight ideas, and for the benefit of readers, we include a brief sketch below.

We compute, for a harmonic function $f : V \to \R$,
\begin{eqnarray*}
  \| f\|^2_{\mathcal H_E}  &=& \frac{1}{2} \sum_{x,y \in V} c_{xy}
(f(x) - f(y))^2.  \\
   &=& \frac{1}{2} \sum_{x\in V} \left[c(x) f^2(x) - 2f(x)\sum_{y\sim x}c_{xy}f(y) + \sum_{y\sim x}c_{xy}f^2(y)\right]  \\
   &=&  \frac{1}{2} \sum_{x\in V} c(x) \left[ - f^2(x) + \sum_{y\sim x}p_{xy} f^2(y)\right]\\
   &=& \frac{1}{2}\sum_{x \in V} c(x) ((Pf^2)(x) - f^2(x)).
\end{eqnarray*}
For the other relation we have
\begin{eqnarray*}
  -\frac{1}{2} \sum_{x \in V} (\Delta f^2)(x) &=& - \sum_{x \in V}\sum_{y \in V} c_{xy} (f^2(x) - f^2(y)\\
   &=& - \frac{1}{2} \sum_{x \in V}\left[c(x) f^2(x) - \sum_{y \sim x}  c_{xy} f^2(y) \right]\\
   &=& - \frac{1}{2}\sum_{x \in V}  c(x) \left[ f^2(x) - \sum_{y \sim x}  \frac{c_{xy}}{c(x)} f^2(y) \right]\\
   &=& \frac{1}{2}\sum_{x \in V} c(x) ((Pf^2)(x) - f^2(x)).
\end{eqnarray*}

Statement (ii) is an obvious generalization of (i).

\end{proof}

\begin{corollary}\label{cor harm in l-1}
A harmonic function $f$ on $(G,c)$ has finite energy if and only if the function $x \mapsto P(f^2)(x) - f^2(x)$ belongs to $l^1(V,c)$.
\end{corollary}

We discuss now an application of the main result of \cite{Ancona_Lyons_Peres1999}. Suppose the following objects are given: a transient network $(G,c)$ with the matrix of transition probabilities $P = (p(x,y))$ that defines the Markov chain $(X_n)$ and the  probability path space $(\Omega, \mathbb P)$, As above, we denote by $(\Omega_x, \mathbb P_x)$ the probability measure space that is formed by all paths starting with $x\in V$ and the corresponding Markov measure. It was proved in \cite{Ancona_Lyons_Peres1999} that, given a function $f \in \h_E$, the sequence $(f\circ X_n)$ converges a.e. and in $L^2$ on the space $(\Omega_x, \mathbb P_x)$ for any $x$. Let $\wt f(\omega) = \lim_{n\to\infty} (f\circ X_n(\omega))$ defined a.e.

Let $\sigma : \Omega \to \Omega ; \sigma(\omega_0, \omega_1, ... ) = (\omega_1, \omega_2, ... )$ be the shift.

\begin{lemma}\label{lem measure P_x} In the above notation, $\Omega_x$ is represented as the disjoint union $\bigcup_{y\sim x} (\{x\} \times \Omega_y)$, and the measures $\mathbb P_x$ and $\mathbb P_y$ are related as follows:
\be\label{eq measures mathbb P_x and mathbb P_y}
d\mathbb P_x(\omega) = \sum_{y \sim x}  p(x,y) d\mathbb P_y(\sigma(\omega)),
\ee
for a.e. $\omega \in \Omega_x$.
\end{lemma}

\begin{proof} It is obvious that if $y_1$ and $y_2$ are distinct neighbors of $x$, then the sets $(\{x\} \times \Omega_{y_1}$ and $(\{x\} \times \Omega_{y_2}$ do not intersect. Thus relation (\ref{eq measures mathbb P_x and mathbb P_y}) should be checked for a single neighbor $y$ of $x$. This result follows easily for $\omega$ belonging to any cylinder set $[x, y, \omega_2, ..., \omega_k]$, and then it can be extended to any Borel set.

\end{proof}

\begin{theorem}\label{thm HF from ALP}
Suppose that $\wt f \in L^1(\Omega_x, \mathbb P_x)$ for every $x \in V$. Then
$$
f(x) = \int_\Omega \wt f(\omega) d \mathbb P_x(\omega)
$$
is harmonic on $(G,c)$ if and only if
$$
\wt f(\omega) = \wt f(\sigma (\omega)).
$$
\end{theorem}

\begin{proof} We need to show that $f = Pf$ if and only the above condition holds. We compute
\begin{eqnarray*}
  f(x) = \sum_{y \sim x} p(x,y) f(y) & \Longleftrightarrow  & \\
  \int_{\Omega_x} \wt f(\omega) d\mathbb P_x(\omega) = \sum_{y \sim x} p(x,y) \int_{\Omega_y} \wt f (\gamma) d\mathbb P_y(\gamma) & \Longleftrightarrow  & \\
  \int_{\Omega_x} \wt f(\omega) d\mathbb P_x(\omega) = \int_{\Omega_x} \wt f(\sigma(\omega)) \sum_{y \sim x} p(x,y) d\mathbb P_y(\sigma(\omega))
\end{eqnarray*}
and the last formula is an  identity because of Lemma \ref{lem measure P_x} and the assumption of the theorem.

\end{proof}

\subsection{Properties of monopoles and dipoles}\label{rem dipole monopole}

Let $x, y$ be arbitrary distinct vertices of an electrical network $(G,c)$. Define the linear functional  $L = L_{xy} : \mathcal H_E \to \R$ by setting $L(u) = u(x) - u(y)$.
It can be shown using connectedness of $G$ that $|L(u)| \leq k\|u\|_{\mathcal H_E}$ where $k$ is a constant depending on $x$ and $y$. By the Riesz theorem, there exists a unique element $v_{xy} \in \mathcal H_E$ such that
\be\label{eq dipole v_xy}
\langle v_{xy}, u \rangle_{\h_E} = u(x) - u(y).
\ee
This element $v_{xy}$ is called a {\em dipole}.  If $o$ is a fixed vertex from $V$, we will use the notation $v_x$ instead of $v_{xo}$. Since, for any $u$, $\langle v_{xy} , u\rangle = \langle v_{x} , u\rangle - \langle v_{y} , u\rangle$, we see that $v_{xy} = v_x - v_y$, and it suffices to study function $v_x, x \in V$, only.  We notice that for any network $(G,c)$ a dipole $v_x$ is always in $\h_E$, and moreover the set $\{v_x : x \in V\}$ is dense in $\h_E$.

The uniqueness of the dipole $v_{xy}$ in $\h_E$ allows one to define a distance in $V$ (see, e.g. \cite{Jorgensen_Pearse2011}):

\begin{lemma}\label{lem resistance dist}
Set, for any $x,y \in V$,
$$
\mathrm{dist}(x, y) = ||v_{xy}||^2_{\h_E}.
$$
Then $\mathrm{dist}(x, y)$ is a metric on $V$, which is called the resistance distance.
\end{lemma}

 By definition, a {\em monopole} at $x \in V$ is an element $w_x \in \h_E$  such that
\be\label{eq-def for w_x}
\langle w_{x}, u \rangle_{\h_E} = u(x)
\ee
for any $u \in \mathcal H_E$. In contrast to case of dipoles, there are networks  $(G,c)$ that do not have monopoles in $\h_E$. In general, the following classical result holds.

\begin{lemma}\label{lem transience - monopoles}
An electrical network $(G,c)$ is transient if and only if there exists a monopole in $\h_E$.
\end{lemma}

In this connection we refer to the paper \cite{Nash-Williams1959} where it is proved that  transience is equivalent to the existence of a flow to infinity of finite energy. We also refer to \cite[Theorem 2.12]{Woess2000}, where this and other relevant results are discussed.

The roles and properties of dipoles and monopoles can be seen from the following statement.

\begin{proposition}\label{prop prop of energy space}
(1) Let $(G,c)$ be a weighted graph and $o$ a fixed vertex from $V$, and let $v_{x} \in \mathcal H_E$ be a dipole corresponding to a vertex $x\in V$. Then
\be\label{eq dipole for v_x}
\Delta v_x = \delta_x - \delta_o.
\ee
More generally, the dipole $v_{xy}$ satisfies the equation $\Delta v_{xy} = \delta_x - \delta_y$. The set  $\mathrm{span} \{v_x\}$ is dense in $\mathcal H_E$.

(2) For any $x\in V$, the Dirac function $\delta_x$ is in $\h_E$, and
$$
c(x) v_x - \sum_{y\sim x} c_{xy}v_y = \delta_x.
$$

(3) If $w_x$ is a monopole corresponding to $x \in V$, then $\Delta w_x = \delta_x$. Moreover $v_{xy} = w_x - w_y$, $x,y \in V$; thus if a monopole $w_{x_0}$ exists as an element of $\h_E$ for some $x_0$, then $w_x$ exists in $\h_E$ for every vertex $x$.

(4) $\mathcal H_E = \mathcal Fin \oplus \mathcal Harm_0$ where $\mathcal Fin$ is the closure of $\mathrm{span}\{\delta_x\}$ with respect to the norm $\|\cdot\|_{\mathcal H_E}$ and $\mathcal Harm_0 = \mathcal Harm \cap \h_E$.
\end{proposition}

\begin{proof} The proof of these and more results can be found in \cite{Jorgensen_Pearse2011, Jorgensen_Pearse2013}.
\end{proof}

\begin{remark} (1) We observe that, in the space of functions  $u$ on $V$, the solution set of the equation $(\Delta u)(z) = (\delta_x - \delta_y)(z)$ is, in general, infinite because the function $u + h$ satisfies the same equation for any $h \in \h arm$. The meaning of Proposition \ref{prop prop of energy space} (1) is the fact that the dipole $v_{xy}$ from (\ref{eq dipole v_xy}) is a unique solution of this equation if it is considered as an element of the space $\h_E$.

(2) It is worth noting that we will use the same terms, monopoles and dipoles, for functions $w_x$ and $v_x$ on $V$ that satisfy the relations $\Delta w_x = \delta_x$ and $\Delta v_x = \delta_x - \delta_o$, respectively.

\end{remark}

\begin{corollary}
Let $x_0 \in V$ be a fixed vertex. Then $w_{x_0}$ is a monopole if and only if it is a finite energy harmonic function on $V \setminus \{x_0\}$.
\end{corollary}

It is not hard to see that the notions of monopoles and dipoles can be extended to more general classes of functions.

\begin{proposition}\label{prop multipoles} Let $F = \{x_0, ... , x_N\}$ be a finite subset of $V$ with $N+1$ distinct vertices. Let $\alpha_i$ be positive numbers such that $\sum_{i=1}^{N} \alpha_i =1$. Then there exists a unique solution $v = v_{F, \alpha} \in \h_E$ such that
\be\label{eq multipoles}
\langle v, f\rangle_{\h_E} = f(x_0) - \sum_{i=1}^{N} \alpha_i f(x_i)
\ee
hold for all $f \in \h_E$. Moreover, the solution $v$ to (\ref{eq multipoles}) satisfies $$
\Delta v = \delta_{x_0} - \sum_{i=1}^{N} \alpha_i \delta_{x_i}.
$$
\end{proposition}

\begin{proof}
The argument is based on the Riesz' theorem applied to the Hilbert space $\h_E$, and is analogous the proof of existence of dipoles in $\h_E$. Assuming $v$ satisfies  (\ref{eq multipoles}), we verify that
$$
w := \Delta v -  ( \delta_{x_0} - \sum_{i=1}^{N} \alpha_i \delta_{x_i})
$$
satisfies $\langle w, v_{oy} \rangle_{\h_E}$ for all $y\in V\setminus \{o\}$, where $\{v_{oy}\}$ is the system of dipoles.
\end{proof}

\subsection{Green's function, dipoles, and monopoles for transient networks}
We shall need to make use of some facts on monopoles, dipoles, and energy Hilbert space; -- for the benefit of readers, we have included a brief fact summary of what is needed. Systematic accounts, and related,  are in \cite{Jorgensen_Pearse2010, Jorgensen_Pearse2014, Georgakopoulos2010}.

As was mentioned above, the Hilbert space $\h_E$ always contains dipoles (see Remark \ref{rem dipole monopole} and Proposition \ref{prop prop of energy space} for the definition and results). Here we will show how a dipole can be found in the space $\h_E$ by an explicit formula assuming that the electrical network $(G,c)$  is transient.

Let $(G,c) = (V,E,c)$ be an electrical network, and $P = (p(x,y) : x, y \in V)$ is the transition probabilities operator where $p(x,y) = \dfrac{c_{xy}}{c(x)}$. Then $P$ defines a random walk $(X_n)$ on $V$ such that $X_n(\omega) = x_n$ where the sequence $\omega = (x_0, ... , x_n, ...) \in \Omega$. For a fixed vertex $a\in V$, we consider the probability space $(\Omega_a, \mathbb P_a) $ where $\Omega_a$ consists of infinite paths that start at $a$, and $\mathbb P_a$ is the corresponding Markov measure on $\Omega_a$.

We recall a few important definitions and facts from theory of Markov chains (see e.g. \cite{Woess2000, Woess2009}). Let $F$ be a subset of $V$ (we will be primarily interested in the case when $F= \{x_1, ... , x_N\}$  is finite). For a probability space $(\Omega_a, \mathbb P_a)$, define the {\em stopping time}
$$
\tau(F)(\omega) = \min\{n \geq 0 : X_n(\omega) \in F\}
$$
with $\omega \in \Omega_a$. It is obvious that
$$
\{\omega \in \Omega : \tau(F) = k \} = \bigcup_{i=1}^N \{\omega \in \Omega : \tau(\{x_i\}) = k \}.
$$
The {\em hitting time}  is defined by
 $$ T(F) =\min\{n \geq 1 : X_n(\omega) \in F\}.
$$
 If $F= \{x\}$ is a singleton, then we write $\tau(x)$ and $T(x)$ for the stopping and hitting times, respectively.

Let $f^{(n)}(x, y) = \mathbb P_x [\tau(y) = n]$, $u^{(n)}(x,x) = \mathbb P_x [T(x) = n]$, and $p^{(n)} (x,y) = \mathbb P_x[X_n = y]$. Then the following quantities are crucial for the study of Markov chains:
$$
G(x,y) = \sum_{n \in \N_0} p^{(n)} (x,y), \ \ F(x,y) = \sum_{n \in \N_0} f^{(n)} (x,y), \ \ U(x,x) = \sum_{n \in \N} u^{(n)} (x,x)
$$

\begin{remark} We recall that $G(x,y)$ is called the Green's function and the quantity $G(x,y)$ is the expected number of visits of $(X_n)$ to $y$ when the random walk starts at $x$. It is well known that the random walk (or, the network $(G,c)$) is transient if and only if $G(x, y) < \infty$ for any $x,y \in V$ \cite{Yamasaki1979}.
This results was rediscovered in \cite{Jorgensen_Pearse2013}  in the context of monopoles and dipoles. Moreover, it was proved in \cite{Jorgensen_Pearse2013} that if the random walk is transient, then, for every $x\in V$, the function $G(x, \cdot)$ is finite energy, i.e., is in $\h_E$. See also Theorem \ref{thm formula for monopole w_x} below.

\end{remark}

The following properties of these functions are well known (see e.g. \cite{Woess2000}).

\begin{lemma}\label{lem relations on G F U} Let $(G, c) =  (V,E,c)$ be an electrical   network. Then, for any pair of vertices $x,y \in V$,
\be\label{eq G U}
G(x,x) = \frac{1}{1 - U(x,x)},
\ee
\be\label{eq G F}
G(x, y) = F(x,y) G(y,y),
\ee
\be\label{eq U F}
U(x,x) = \sum_{y \sim x} p(x,y) F(y,x),
\ee
\be\label{eq F F}
F(x,y) = \sum_{z \sim x} p(x,z) F(z,y), \ \ (x\neq y).
\ee

\end{lemma}

\begin{remark}\label{rem reversibility F} It follows from the reversibility of the Markov chain $(X_n)$ defined by transition probabilities $P = (p(x,y))$ that for the functions $F(x,y)$ and $G(x,y)$ satisfy the properties:
$$
c(x) F(x, y) = c(y)F(y, x) \ \  \ \mathrm{and} \ \ \ c(x) G(x, y) = c(y)G(y, x).
$$
Moreover, since the quantity $F(x,y)$ can be treated  as the probability of the event
that the random walk starting at $x$ reaches $y$, we can write $F(x, y) =\mathbb Pr[x \to y]$.

\end{remark}

We will frequently use the following statement.

\begin{lemma}\label{lem h_1 in terms of F}
Let $F$ be a subset of $V$ and let $x\in F$ be any fixed vertex. We define
$$
h_x(a) := \mathbb E_a(\chi_{\{x\}}\circ X_{\tau(F)}).
$$
Then
$$
h_x(a) =  \mathbb E_a(\chi_{\{x\}}\circ X_{\tau(x)}) = F(a, x), \ \ \ \forall a\in V.
$$
\end{lemma}
\begin{proof}
The first equality follows directly from the fact that $x\in F$, and the second one is due to the obvious observation that
$$
\int_{\Omega_a} \chi_{\{x\}}( X_n(\omega)) d\mathbb P_a(\omega) = f^{(n)}(a, x).
$$
\end{proof}

In the next two lemmas, we discuss the other properties of the function $h_x$ which will also be in use below.

\begin{lemma} \label{lem properties of h_i}
Given a subset $F = \{x_1, ... , x_N\}$ of vertices from $V$, we set
$$
h_i(a) := \mathbb E_a(\chi_{\{x_i\}}\circ X_{\tau(F)}) = \int_{\Omega_a} \chi_{\{x_i\}}(X_{\tau(F)}(\omega)) d\mathbb P_a(\omega), \ i =1, ... ,N,
$$
where $a$ is an arbitrary vertex in $V$. Then
$$
(\Delta h_i)(a) = 0, \ \ a\in V\setminus F, \ \ \mbox{and}\ \ h_i(x_j) = \delta_{ij}.
$$
In other words,
$$
h_i(a) = \begin{cases} 1, \ & a  = x_i\\
                                                 0, \ &  a \in F \setminus \{x_i\}\\
                                                 \mbox{\rm{harmonic}}, \ &  a \in V \setminus F.

\end{cases}
$$
\end{lemma}

\begin{proof} It follows from the definition of $h_i, i=1, ... ,N,$ that, since $F(x_i, x_i) =1$, we see that  $h_i(x_i) =1$, and $h_i(x_j) =0$ for $j\neq i$.

Fix a vertex $a \in V \setminus F$ and show that $(\Delta h_i)(a) = 0$. Equivalently, we verify that $Ph(a) = h(a)$:
\begin{eqnarray*}
h_i(a)  &=& \sum_{b\sim a} p(a, b) \mathbb E_a(\chi_{\{x_i\}} \circ X_{\tau(F)} | X_1 = b)\\
   &=&  \sum_{b\sim a} p(a, b) \mathbb E_b( \chi_{\{x_i\}} \circ X_{\tau(F)} )\ \ \mbox{(by\ the\ Markov\ property}) \\
   &=&  \sum_{b\sim a} p(a, b) h_i(b)\\
   & = & (Ph_i)(a)
\end{eqnarray*}
\end{proof}

\begin{corollary}\label{cor interpolation}
Let $F= \{x_1, ... , x_N\}$ be a finite subset of $V$, and $h_{x_i}$ is defined as in Lemma \ref{lem h_1 in terms of F}. Suppose that $\varphi : F \to \R$ is a given function on $F$. Define
\begin{equation}\label{eq interpolation}
\Phi(a) := \sum_{i=1}^N \varphi(x_i) h_{x_i}(a).
\end{equation}
Then $\Phi$ is a solution of the Dirichlet problem
$$
\begin{cases} (\Delta \Phi)(a)= 0, & a \in V\setminus F,\\
\Phi(a) = \varphi(a), & a \in F.
\end{cases}
$$
\end{corollary}

\begin{proof}
Indeed, this can be seen from Lemmas \ref{lem h_1 in terms of F} and \ref{lem properties of h_i} because $h_{x_i}(x_j) = \delta_{ij}$ and $h_{x_i}$ is harmonic on $V \setminus F$. Furthermore, relation (\ref{eq interpolation}) represents an interpolation formula for a given function $\varphi$.
\end{proof}

It is worth noticing that we have not used so far our assumption about transience of $(G,c)$. Just based on the definition of the function $h_x$, we can give an upper bound for the energy of $h_x$. It will be proved in Theorem \ref{thm formula for monopole w_x}  that the energy of $h_x$ is finite for a transient network $(G,c)$.

\begin{lemma}\label{lem energy of h_x} Let $h_x$ be defined as in Lemma \ref{lem h_1 in terms of F}. Then
$$
\|h_x\|_{\h_E}^2  < \frac{1}{2} c(x) \sum_{a\in V} \mathbb Pr[x \to a](1 - \mathbb Pr[a \to x]).
$$
\end{lemma}

\begin{proof} We use the equalities $c_{ab} = c(a)p(a,b)$ and $\sum_b p(a,b)F(b,x) = F(a,x)$ (Lemma \ref{lem relations on G F U}, and the inequality $F(b,x)^2 < F(b,x)$ in order to calculate the energy of $h_x$:

\begin{eqnarray*}
  \|h_x\|_{\h_E}^2  &=&  \frac{1}{2} \sum_{a,b} c_{ab}(h_x(a) - h_x(b))^2\\
  &=& \frac{1}{2} \sum_{a,b} c_{ab}(\mathbb E_a(\chi_{\{x\}} \circ X_{\tau(x)}) -
\mathbb E_b(\chi_{\{x\}} \circ X_{\tau(x)}))^2 \\
   &=&   \frac{1}{2} \sum_{a,b} c_{ab}(F(a,x) - F(b, x))^2 \\
   &=& \frac{1}{2} \sum_{a}\left[ c(a)F(a,x)^2 - 2F(a,x)\sum_{b\sim a}c_{ab}F(b, x) +  \sum_{b\sim a}c_{ab}F(b, x)^2\right] \\
   &=& \frac{1}{2} \sum_{a}\left[ - c(a)F(a,x)^2  + \sum_{b\sim a}c(a)p(a,b)F(b, x)^2\right]\\
  & < & \frac{1}{2} \sum_{a}\left[ - c(a)F(a,x)^2  + c(a)F(a,x)\right] \\
  &=& \frac{1}{2} c(x) \sum_{a\in V} \mathbb Pr[x \to a](1 - \mathbb Pr[a \to x]).
\end{eqnarray*}
The last equality follows from Remark \ref{rem reversibility F}.
\end{proof}

From now on, we focus on the case when the network $(G,c)$ is {\em transient}, and $N=2$. Our goal is to find a formula for $v_{x_1x_2}$ solving the equation $\Delta v_{x_1, x_2} = \delta_{x_1} - \delta_{x_2}$ in the space $\h_E$ for any fixed $x_1, x_2\in V$.  Then, as was stated in Proposition \ref{prop prop of energy space}, the function $v_{x_1x_2}$ will be a dipole. Simultaneously, we will find a formula for  a monopole $w_x \in \h_E$ at $x \in V$ satisfying the equation $\Delta w_x = \delta_x$.

Suppose that $F = \{x_1, x_2\}$, and $x_1\neq x_2$. Let  $h_1$ and $h_2$ be the function defined in Lemma \ref{lem properties of h_i}. Consider the matrix
\be\label{eq def M}
M:= \left(
  \begin{array}{cc}
   (\Delta h_1)(x_1)  &  (\Delta h_2)(x_1) \\
\\
     (\Delta h_1)(x_2) &  (\Delta h_2)(x_2) \\
  \end{array}
\right)
\ee
and compute its entries using Lemma \ref{lem h_1 in terms of F}. To do this, we apply the relation $\Delta h  = c( I -P) h$, which is valid for any function $h$ on $V$. We first find the off-diagonal entries for $i \neq j$:
\begin{eqnarray*}
  (\Delta h_i)(x_j) &=& c(x_j)[(I - P)h_i](x_j)\\
   &=& c(x_j) [h_i(x_j) - \sum_{y \sim x_j} p(x_j,y) h_i(y)] \\
   &=& c(x_j) [ - \sum_{y \sim x_j} p(x_j,y) \mathbb E_y(\chi_{\{x_i\}}
\circ X_{\tau(F)})]\\
& = &  - c(x_j) \sum_{y \sim x_j} p(x_j,y) \mathbb E_y(\chi_{\{x_i\}}
\circ X_{\tau(x_i)})\\
 &= & - c(x_j)  \sum_{y \sim x_j} p(x_j,y) F(y, x_i)  \\
  &=& - c(x_j) [ F(x_j, x_i) ].
\end{eqnarray*}
We have used here the fact that $h_i(x_j) =\delta_{ij}$, relation (\ref{eq F F}), and
Lemma \ref{lem h_1 in terms of F}.

Similarly, we compute the diagonal entries for $i =j$:
\begin{align}\label{eq diagonal entry}
\begin{split}
(\Delta h_i)(x_i) &=  c(x_i)[(I - P)h_i](x_i)
\\
  & = c(x_i) [h_i(x_i) - \sum_{y \sim x_i} p(x_i,y) h_i(y)]
\\
& = c(x_i)[1 - \sum_{y \sim x_i} p(x_i,y) \mathbb E_y(\chi_{\{x_i\}}
\circ X_{\tau(x_i)})]
\\
& =c(x_i) (1 - U(x_i, x_i)),\\
\end{split}
\end{align}
where (\ref{eq U F}) is used.

Thus, we have proved the first assertion  of the next lemma, stating that the matrix $M$ admits the following factorization.

\begin{lemma}\label{lem matrix M}
The matrix $M$ defined in (\ref{eq def M}) is represented as follows:
\be
M = \left(
      \begin{array}{cc}
        c(x_1) & 0 \\
\\
        0 & c(x_2) \\
      \end{array}
    \right) \left(
              \begin{array}{cc}
                1 - U(x_1, x_1) &  - F(x_1, x_2) \\
\\
                 - F(x_2, x_1) & 1 - U(x_2, x_2) \\
              \end{array}
            \right).
\ee
Moreover,
\be
\det M = \frac{c(x_1)c(x_2)( 1 - G(x_1, x_2)G(x_2, x_1))}{G(x_1, x_1)G(x_2, x_2)}
\ee
and
$$
\det M = 0\ \ \Longleftrightarrow \ \ G(x_1, x_2) = \sqrt{\frac{c(x_2)}{c(x_1)}}.
$$
\end{lemma}

\begin{proof} We first notice that, by Remark \ref{rem reversibility F},  the off-diagonal entries in $M$ are equal. It remains to show that
$$
\det \left(
              \begin{array}{cc}
                1 - U(x_1, x_1) &  - F(x_1, x_2) \\
                 - F(x_2, x_1) & 1 - U(x_2, x_2)
              \end{array}
     \right) = \frac{ 1 - G(x_1, x_2)G(x_2, x_1)}{G(x_1, x_1)G(x_2, x_2)}.
$$
Indeed, applying Lemma \ref{lem relations on G F U}, we obtain
\begin{eqnarray*}
(1 - U(x_1, x_1))( 1 - U(x_2,x_2))& - &F(x_1, x_2)F(x_2, x_1)\\
\\
 &= & \frac{1}{G(x_1, x_1)G(x_2, x_2)} - F(x_1, x_2)F(x_2, x_1)\\
\\
   & =&  \frac{ 1 -F(x_1, x_2)F(x_2, x_1) G(x_1, x_1)G(x_2, x_2)}{ G(x_1, x_1)G(x_2, x_2)}\\
\\
   &=& \frac{ 1 - G(x_1, x_2)G(x_2, x_1)}{G(x_1, x_1)G(x_2, x_2)}.
\end{eqnarray*}

The last statement of the lemma is based on the  fact that the Markov chain is reversible, see Remark \ref{rem reversibility F}, that is
\be\label{eq for G(x,y)}
c(x_1) G(x_1,x_2) = c(x_2)G(x_2,x_1), \ \ \ x_1, x_2 \in V.
\ee
  We conclude that $\det M = 0$ if and only if $G(x_1, x_2)G(x_2, x_1) = 1$, if and only if the pair $(x_1, x_2)$ is degenerate in the following sense:
$$
G(x_1, x_2) = \sqrt{\frac{c(x_2)}{c(x_1)}}.
$$

\end{proof}

Now we are ready to prove our main results of this section.

\begin{theorem}\label{thm formula for monopole w_x}
Let $(G,c)$ be  a transient network, and $x$ a fixed vertex in $V$. Let $h_x$ be the function defined in Lemma \ref{lem h_1 in terms of F}, i.e.,
$$
h_x(a) := \mathbb E_a(\chi_{\{x\}} \circ X_{\tau(x)}), \ \ a \in V.
$$
Then the function
\be\label{eq def of w_x}
w_x(a) :=  \frac{G(a,x)}{c(x)}, \ \ a \in V,
\ee
is a monopole at $x$. In other words,  $w_x\in \h_E$ and it satisfies the equation $\Delta w_x = \delta_x$.

\end{theorem}

\begin{proof}  We first show that $w_x$ is a multiple of $h_x$. It follows from Lemma \ref{lem h_1 in terms of F} and (\ref{eq G F}) that
\begin{eqnarray*}
w_x(a) & = & \frac{G(a,x)}{c(x)} \\
   &=& \frac{F(a,x)G(x,x)}{c(x)}\\
   &=& \frac{1}{c(x)(1 - U(x,x))}h_x(a).
\end{eqnarray*}
To prove the theorem, it suffices to show that
$$
 (i)\ \  (\Delta w_x)(a) = \delta_x(a), \ \ \ \mathrm{and} \ \ \ (ii) \ \
\langle w_x, f\rangle_{\h_E} = f(x), \ \ \forall f \in \h_E.
$$
In (\ref{eq diagonal entry}) and Lemma \ref{lem properties of h_i} we showed that
$$
(\Delta h_x)(a) = \begin{cases}
\dfrac{c(x)}{G(x,x)}, & a =x\\
0, & a\neq x.
\end{cases}
$$
Hence, (i) is proved.

To see that (ii) holds, we compute, for $f\in \h_E$,
\begin{eqnarray*}
\langle w_x, f\rangle_{\h_E}  &=&  \frac{1}{2}\sum_{a,b \in V}c_{ab} (w_x(a) - w_x(b)) (f(a) - f(b)) \\
   &=& \frac{1}{2}\sum_{a\in V}\left[\sum_{b \in V}c_{ab} (w_x(a) - w_x(b))f(a) \right]  \\
   & + & \frac{1}{2}\sum_{b\in V}\left[\sum_{a \in V}c_{ab} (w_x(b) - w_x(a)f(b)\right] \\
   &=& \frac{1}{2}\sum_{a\in V}  (\Delta w_x)(a)f(a) + \frac{1}{2}\sum_{b\in V}  (\Delta w_x)(b)f(b) \\
   &=& \frac{1}{2} f(x) +  \frac{1}{2} f(x)\\
   &=& f(x).
  \end{eqnarray*}
The theorem is proved.
\end{proof}

\begin{remark} (1)  It follows from Theorem \ref{thm formula for monopole w_x} that  the monopole $w_x$ has finite energy and
$||w_x||_{\h_E} = \dfrac{G(x,x)}{c(x)} \|h_x\|_{\h_E}$.

(2) Moreover, we deduce from relation (\ref{eq def of w_x}) the following result:  an electrical network $(G,c)$ is transient if and only if the function $a \mapsto G(a,x)$ has finite energy for every fixed $x\in V$.
\end{remark}

\begin{corollary}\label{cor dipole formula}
Let $x_1, x_2$ be any distinct vertices in $V$. Set
$$
v_{x_1, x_2}(a) := (w_{x_1} - w_{x_2})(a) = \frac{G(a, x_1)}{c(x_1)}  -  \frac{G(a, x_2)}{c(x_2)},
$$
 Then $v_{x_1, x_2}$ is a dipole in $\h_E$, i.e., $(\Delta v_{x_1, x_2})(a) = (\delta_{x_1} - \delta_{x_2}) (a)$, and it has finite energy.
\end{corollary}

\begin{proof}
From the definition of $v_{x_1,x_2}$, we see that
$$
v_{x_1,x_2} = \frac{G(x_1, x_1)}{c(x_1)} h_1 + \frac{G(x_2, x_2)}{c(x_2)} h_2.
$$
Then by Theorem \ref{thm formula for monopole w_x}, we have
$$
\Delta v_{x_1,x_2} = \Delta w_{x_1} - \Delta w_{x_2} = \delta_{x_1} - \delta_{x_2}.
$$
 Moreover,
$$
\langle v_{x_1,x_2}, f\rangle_{\h_E} = \langle w_{x_1}, f \rangle_{\h_E}  -  \langle w_{x_2}, f \rangle_{\h_E} = f(x_1) - f(x_2).
$$
This proves that  the dipole $v_{x_1,x_2}$ belongs to the energy space $\h_E$.
\end{proof}

We finish this section with another result regarding dipoles.

\begin{theorem}\label{thm dipole formula}
Let $(G, c)$ be  a transient electrical network, and let $x_1, x_2$ be any two distinct vertices in $V$ such that the Green's function $G$ (see Lemma \ref{lem relations on G F U}) satisfies the relation
$$
G(x_1, x_2) \neq \sqrt{\frac{c(x_2)}{c(x_1)}}.
 $$
 Let $M$ be the matrix defined by (\ref{eq def M}). Then the function
$$
\ol v_{x_1, x_2}(a) = \alpha h_1(a) + \beta h_2(a), \ \ a\in V,
$$
is a dipole defined on $V$, where the coefficients $\alpha$ and $\beta$ are determined as  the solution to the equation
\be\label{eq sol for M}
M\left(
   \begin{array}{c}
     \alpha \\
     \beta \\
   \end{array}
 \right) = \left(
   \begin{array}{r}
   1\\
    -1\\
   \end{array}
 \right)
\ee
Equivalently, we can write
$$
\ol v_{x_1, x_2}(a) = [h_1(a), h_2(a)]M^{-1}\left(
   \begin{array}{r}
   1\\
    -1\\
   \end{array}
 \right), \qquad a\in V.
$$
\end{theorem}

\begin{proof} It follows from Lemma \ref{lem matrix M} that the assumption of the theorem implies that (\ref{eq sol for M}) has a unique solution.

We need to verify that  the function $\ol v_{x_1, x_2}$ satisfies
\be\label{eq for ol v}
\Delta \ol v_{x_1, x_2} = \delta_{x_1} - \delta_{x_2}.
\ee
 Indeed, one has
$$
(\Delta \ol v_{x_1, x_2})(a) = \alpha (\Delta h_1)(a) + \beta (\Delta h_2)(a).
$$
Substituting  $a = x_1$ and $a = x_2$ in the above equality, we obtain
$$
\alpha (\Delta h_1)(x_1) + \beta (\Delta h_2)(x_1) = 1,
$$
$$
\alpha (\Delta h_1)(x_2) + \beta (\Delta h_2)(x_2) = -1
$$
since the vector $(\alpha, \beta)^T$ was chosen satisfying equation (\ref{eq sol for M}). We recall that, by Lemma \ref{lem properties of h_i}, the functions $h_1$ and $h_2$ are harmonic on $V\setminus F$. Hence, (\ref{eq for ol v}) is proved.

\end{proof}

\begin{remark} It follows from Corollary \ref{cor dipole formula} and Theorem \ref{thm dipole formula} that the functions $v_{x_1,x_2}$ and $\ol v_{x_1,x_2}$
satisfy the same equation. Therefore their difference $f =  v_{x_1,x_2} - \ol v_{x_1,x_2}$ is a harmonic function which again is a linear combination of $h_1$ and $h_2$.

\end{remark}

\section{Existence of harmonic functions on a Bratteli diagram}\label{sect HF on BD}

In this section, we will study the space $\mathcal Harm$ of harmonic functions on  arbitrary  Bratteli diagrams. It will be proved that the dimension of  $\mathcal Harm$ depends on the structure of an underlying Bratteli diagram and can be either finite (for a rather restrictive class of stationary Bratteli diagrams) or infinite. Moreover, we will show that Bratteli diagrams $B = (V, E)$ of ``bottleneck'' type have only trivial harmonic functions defined on the set of all vertices $V$. It is worth mentioning that harmonic functions as elements of the energy space $\h_E$ will be  considered in Section \ref{sect HF and energy}.

\subsection{Characterization of harmonic functions}

We first recall our basic settings. Let $B = (V,E)$ be a 0-1 Bratteli diagram with the sequence of 0-1 incidence matrices $A_n$. We assume that the conductance function $c$  is defined on $E$ and takes positive value at every edge $e$. Since every edge $e$ is uniquely determined by a pair of vertices $(x,y)$, we write also $c_e = c_{xy} = c_{yx}$. Based on the structure of the vertex set $V = \coprod_{n\geq 0}V_n$ and edge set $E= \coprod_{n\geq 0}E_n$ of the Bratteli diagram $B$,
we define a sequence of matrices $(C_n)_{n \ge 0}$ which is naturally related to the incidence matrices $(A_n)$ and the conductance function $c$:
$$
C_n = (c^{(n)}_{xy}),
$$
where, by definition,
$$
 c^{(n)}_{xy} := \begin{cases} c_{xy},\  & x =s(e), y=r(e), e \in E_n\\
 0, & \mbox{otherwise}\\
\end{cases}
$$
Then $c^{(n)}_{xy} > 0$  if and only if $a^{(n)}_{xy} =1$ and the size of $C_n$ is $|V_n| \times |V_{n+1}|, n \ge 0$.
In particular, $C_0 $ is a row matrix with entries $(c^{(0)}_{ox} : x \in V_1)$. It is helpful to remember that for every $n$, the matrix $C_n$ determines a linear transformation $C_n : \R^{|V_{n+1}|}  \to \R^{|V_{n}|}$.

We note that the order of indexes in $c^{(n)}_{xy}$ is important: although the values of the conductance function $c$ depend on edges only, the entry  $c^{(n)}_{yx}$ belongs to the transpose $C_n^T$ of $C_n$.

It is said that the sequence of matrices $(C_n)$ is {\em  associated to the weighted Bratteli diagram $(B,c)$}.

Together with the sequence of associated matrices $(C_n)$, we will consider two other sequences of matrices. They are denoted by $(\overleftarrow{P}_n)$ and $(\overrightarrow{P}_{n-1})$, and their entries are defined by the formulas
$$
\overleftarrow{p}_{xz}^{(n)} = \frac{c^{(n)}_{xz}}{c_n(x)}, \ \ x\in V_n, z \in V_{n+1},
$$
$$
\overrightarrow{p}_{xy}^{(n-1)} = \frac{c^{(n-1)}_{yx}}{c_n(x)}, \ \ x\in V_n, y \in V_{n-1}.
$$
This means, in particular,  that $\overleftarrow{P}_0$ is a row matrix, and,  for all $n$, $\overrightarrow{P}_{n} = \overleftarrow{P}_{n}^T$ where $T$ stands for the transpose matrix.

\begin{remark}
(1) In order to clarify the essence of our notation, let us imagine a Bratteli diagram as an infinite graph that is expanding in the ``horizontal'' direction from left to right, that is it starts at the top vertex $o$ and passes consequently  through the ``vertical'' levels $V_n$. Then the arrows used in the notation of the matrices show how the transformations defined by the matrices act: $\overleftarrow{P}_n$ sends $\R^{|V_{n+1}|}$ to $\R^{|V_{n}|}$ and $\overrightarrow{P}_{n-1}$ sends $\R^{|V_{n-1}|}$ to $\R^{|V_{n}|}$, see Figure 3.

(2) The matrix $P$ of transition probabilities has a simple form. It can be schematically represented as follows
$$
P = \left(
  \begin{array}{cccccc}
    0 & \overleftarrow{P}_0 & 0 & 0 & \cdots &\cdots \\
    \overrightarrow{P}_0 & 0 & \overleftarrow{P}_1 & \cdots &\cdots \\
    0 & \overrightarrow{P}_1 & 0 & \overleftarrow{P}_2 & \cdots &\cdots \\
    0 & 0 & \overrightarrow{P}_2 & 0 & \overleftarrow{P}_3 & \cdots  \\
    \cdots &\cdots &\cdots &\cdots  &  \cdots &\cdots \\
  \end{array}
\right).
$$
Here every entry $P_{ij}, i,j = 0,1,2, ... ,$ corresponds a block matrix whose rows are enumerated by vertices from $V_i$ and columns are enumerated by vertices from $V_j$.
\end{remark}

\begin{figure}[ht!]
\begin{center}
\includegraphics[scale=0.75]{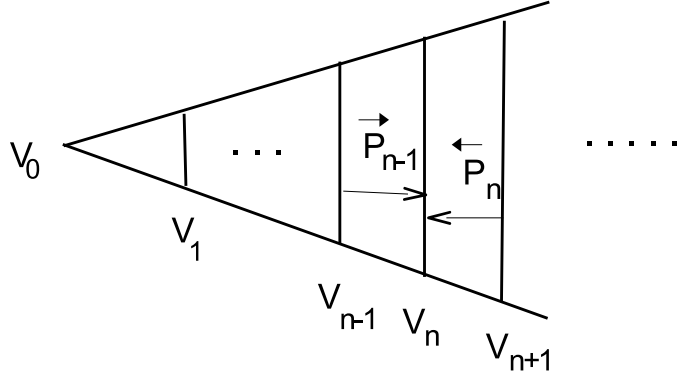}
\caption{Matrices  acting on a Bratteli diagram.}
\label{BD_via_matrices}
\end{center}
\end{figure}

Suppose that $f : V \to \R$ is a harmonic function on a weighted graph $(G,c)$, i.e., $(\Delta f) (x) =0$ for all $x \in V$. Then it follows from (\ref{Laplacian formula}) that $f \in \mathcal Harm$ if and only if
\begin{equation}\label{harm function in c-terms}
c(x) f(x)  = \sum_{y\sim x} c_{xy}f(y), \ \ \forall x \in V.
\end{equation}

If $G$ is a Bratteli diagram $B = (V,E)$, then we can extract more information from relation (\ref{harm function in c-terms}). We first note that any function $f : V \to \R$ can be  uniquely identified with a sequence $(f_n : n \in \N_0)$ of vectors where $f_n = f|_{V_n}\in \R^{|V_{n}|}$. In particular, the function $x \mapsto c(x) : V \to \R$ can be equivalently written as a sequence of vectors $(c_n : n \in \N_0)$. Next, for $x\in V_n$, relation (\ref{harm function in c-terms}) is expanded as
$$
c(x) f(x)  = \sum_{y\in V_{n-1}} c^{(n-1)}_{xy}f(y) + \sum_{z\in V_{n+1}} c^{(n)}_{xz}f(z)
$$
or, in matrix notation,
\be
(D_n f_n)(x) = (C_{n-1}^T f_{n-1})(x) + (C_{n} f_{n+1})(x)
\ee
where  $D_n$ is the diagonal matrix with $d^{(n)}_{xx} = c_n(x), x \in V_n$.

On the other hand, if $x\in V_n$, then $c(x) = c_n(x)$ can be represented as follows:
$$
c(x) = \sum_{y\in V_{n-1}} c^{(n-1)}_{xy} + \sum_{z\in V_{n+1}} c^{(n)}_{xz} =
(C_{n-1}^T \mathbf {1}_{n-1})(x) + (C_{n} \mathbf {1}_{n+1})(x)
$$
where $\mathbf {1}_n = (1, ... , 1)^T \in \mathbb R^{|V_n|}$. Finally, we conclude  from these relations that, for any $x\in V_n$ and any $n$,
\begin{equation}\label{C_n relations with x for harm f-n}
C_{n-1}^T (f_{n-1} - f_n(x)\mathbf {1}_{n-1})(x) + C_{n}(f_{n+1} - f_n(x) \mathbf {1}_{n+1})(x) = 0.
\end{equation}

We summarize the above argument in the following statement.

\begin{theorem}\label{thm ness-suff matrix cond for harm fns}
Let $(B, c) = (V,E, c)$ be a weighted 0-1 Bratteli diagram with conductance function $c$. Let $(C_n)$ be  the sequence of matrices  associated to $(B,c)$. Then a function $f : V\to \R$ is harmonic if and only the sequence of vectors $(f_n)$ where $f_n = f|_{V_n}$ satisfies
\be\label{D_n and C_n represents harm fn}
C_{n} f_{n+1} = D_n f_n - C_{n-1}^T f_{n-1},\ \ \ n\in \N,
\ee
or, equivalently, relation (\ref{C_n relations with x for harm f-n}) holds.
\end{theorem}

\begin{proof}
The assertion ``if'' means that when a harmonic function $f = (f_n)$ is given, then the vectors $(f_n)$ must satisfy (\ref{D_n and C_n represents harm fn}). This fact has been already proved above. For ``only if'', suppose that a sequence of vectors  $f_n \in \R^{|V_n|}, n\ge 1,$ satisfying (\ref{C_n relations with x for harm f-n}) exists. Then it is again straightforward to verify that the corresponding function $f$ is harmonic, i.e., (\ref{harm function in c-terms}) holds.
\end{proof}

We address now the question about the existence of  solutions of  (\ref{C_n relations with x for harm f-n}) and (\ref{D_n and C_n represents harm fn}). Recall that we can always assume, without loss of generality, that $f(o) =0$ for any harmonic function $f = (f_n)$ defined on a Bratteli diagram with the root $o$. Our first goal is to find conditions on vectors $(f_n)$ under which $(f_n)$ represents a harmonic function. Since $f(o) = 0$, we obtain that the vector $f_1$ must satisfy the equation
\begin{equation}\label{harm f-n on V_0}
C_0^T f_1 =0 \ \ \ \mbox{or} \ \ \ \sum_{x \in V_1} c^{(0)}_{ox} f_1(x) = 0.
\end{equation}
Denote the solution set of (\ref{harm f-n on V_0}) by $\mathcal N_1$. Then $\mathcal N_1$ is a $(|V_1| -1)$-dimensional subspace of $\R^{|V_1|}$.

Next, it follows from (\ref{C_n relations with x for harm f-n}) that, for any $f_1 \in \mathcal N_1$ and a vertex $x \in V_1$,
\begin{equation}\label{harm f-n on V_1}
C_0^T(-f_1(x)\mathbf 1_0)(x) + C_1(f_2 - f_1(x)\mathbf 1_2)(x) =0.
\end{equation}
Since $\mathbf 1_0 =1$, relation  (\ref{harm f-n on V_1}) is transformed into
$$
- c_{ox}^{(0)} f_1(x) + (C_1f_2)(x) - f_1(x) \sum_{z\in V_2}c^{(1)}_{xz}  = 0.
$$
Hence, using (\ref{harm function in c-terms}),  we obtain that $(C_1f_2)(x) = c_1(x)f_1(x)$ for arbitrary $x\in V_1$ where $c_1 = c|_{V_1}$. Thus,
\be\label{harm f-n on V_1 short form}
C_1f_2 = D_1f_1 \ \ \ \mbox{or\ equivalently, } \ \ \ \overleftarrow{P}_1 f_2 = f_1.
\ee

It follows that a solution of  (\ref{harm f-n on V_1 short form}) exists if and only if  the vector $f_1 = (f_1(x) : x \in V_1)$ belongs to $\mbox{Col} (\overleftarrow{P}_1)$, the column space of the matrix $\overleftarrow{P}_1$.

We take one more step to see what happens in the general situation. If $x\in V_2$, then equation (\ref{C_n relations with x for harm f-n}) is
$$
(C_2f_3)(x) = f_2(x)(C_2 \mathbf 1_3)(x) + f_2(x) (C_1^T\mathbf 1_1)(x) - (C_1^Tf_1)(x).
$$
The latter can  also be written  in vector form
\be\label{harm f-n on V_2 short form}
\overleftarrow{P_2}f_3 = f_2 - \overrightarrow{P}_1f_1.
\ee
Thus, we conclude that (\ref{harm f-n on V_2 short form}) has a solution if and only if $f_2 - \overrightarrow{P}_1f_1$ belongs to $\mbox{Col} (\overleftarrow{P}_2)$. We notice that formula (\ref{harm f-n on V_1 short form}) is of the same nature as  (\ref{harm f-n on V_2 short form})    because $f_0 = 0$.

It follows from the above arguments that  the following statement (which is a corollary of  Theorem \ref{thm ness-suff matrix cond for harm fns})  holds. We illustrate this assertion in Figure 3.

\begin{corollary}\label{cor existence of harmonic f-n}
Let $(B(V,E), c)$ be a weighted Bratteli diagram with associated sequences of matrices $(\overrightarrow{P}_n)$ and $(\overleftarrow{P}_n)$. Then a sequence of vectors $(f_n)$ ($f_n \in \R^{|V_n|}$) represents a harmonic function $f  = (f_n) : V \to \R$ if and only if for any $n\geq 1$
\be\label{eq ness-suff for harmonic f-n}
f_n - \overrightarrow{P}_{n-1} f_{n-1} = \overleftarrow{P}_n f_{n+1}.
\ee
\end{corollary}

This results allows us to formulate one more corollary:

\begin{corollary}\label{cor existence of harmonic f-n-1} In notation of Corollary \ref{cor existence of harmonic f-n}, the space of harmonic functions, $\h arm$, is nontrivial on a weighted Bratteli diagram $(B, c)$  if and only if there exists a sequence of non-zero vectors $f =(f_n)$, where $f_n \in \R^{|V_n|}$, such that

\be\label{eq ness-suff for harmonic f-n}
f_n - \overrightarrow{P}_{n-1} f_{n-1} \in \mbox{Col} (\overleftarrow{P}_n).
\ee

\end{corollary}

\subsection{Algorithmic construction of harmonic functions on Bratteli diagrams}\label{subsect algorithm}

Corollary \ref{cor existence of harmonic f-n-1} is used for formulation  an {\em algorithm} for finding conditions on a weighted Bratteli diagram $(B,c)$ that would guarantee the existence of nontrivial harmonic functions on  $(B, c)$. Moreover, we also give some simplified sufficient conditions for which relation (\ref{eq ness-suff for harmonic f-n}) holds.
\\

\textbf{\textit{Algorithm}}:

(I) Find $\mathcal N_1$, the solution set of $C_0 f_1 =0$.

(II) Find $\mathcal N_2 = \{f_2 \in \R^{|V_2|} : \overleftarrow{P}_1 f_2 \in \mathcal N_1\}$. If $\mbox{Col}(\overleftarrow{P}_1)\cap \mathcal N_1 =\{0\}$ (this is possible only if $\mbox{dim}(\mbox{Col})(\overleftarrow{P}_1) = 1$), then the space $\mathcal N_2$ is trivial, and this means that $\mathcal Harm$ is trivial.

(III) Consider the space $\mathcal G_2 = \{ f_2 - \overrightarrow{P}_1 f_1 \in \R^{|V_2|} :  f_2 \in \mathcal N_2,\ f_1 \in \mathcal N_1\}$. Let $\mathcal N_3 = \{ f_3 \in \R^{|V_3|} : \overleftarrow{P}_2 f_3 \in \mathcal G_2\}$. If $\mathcal N_3$ is trivial, i.e., $\mbox{Col}(\overleftarrow{P}_2) \cap \mathcal G_2 = \{0\}$, we have to stop the algorithm.

(IV) The general case, when $n$ is arbitrary, repeats step (III). We define $\mathcal N_{n+1} = \{ f_{n+1} \in \R^{|V_{n+1}|} : \overleftarrow{P}_{n} f_{n+1} \in \mathcal G_{n}\}$ where $\mathcal G_{n} = \{ f_{n} - \overrightarrow{P}_{n-1} f_{n-1} \in \R^{|V_{n}|} :  f_{n} \in \mathcal N_{n},\ f_{n-1} \in \mathcal N_{n-1}\}$. If for every $n$ the space $\mbox{Col}(\overleftarrow{P}_n) \cap \mathcal G_n \neq \{0\}$, then any nontrivial solution of the equation $ \overleftarrow{P}_{n} f_{n+1} = g_n$ represents a non-constant harmonic function, where $g_n \in \mathcal G_n$.

Thus, we conclude that the following statement holds.

\begin{proposition} \label{prop suff cond for exist harm fns}
The space $\mathcal Harm$ of  harmonic functions on a weighted Bratteli diagram $(B,c)$ is nontrivial if and only if  for every $n$
\be\label{suff cond notrivial harm fns}
\mbox{Col}(\overleftarrow{P}_n) \cap \mathcal G_n \neq  \{0\}
\ee
where $\mathcal G_{n} = \{ f_{n} - \overrightarrow{P}_{n-1} f_{n-1} \in \R^{|V_{n}|} :  f_{n} \in \mathcal N_{n},\ f_{n-1} \in \mathcal N_{n-1}\}$, and  $\mathcal N_{n}$ is defined above.
In particular, if $\mbox{Rank}(\overleftarrow{P}_n) = |V_n|$ for all $n\geq 1$ then (\ref{suff cond notrivial harm fns}) is automatically satisfied.

\end{proposition}

There are several obvious corollaries that follow from  Proposition \ref{prop suff cond for exist harm fns}. We discuss them in the following remark.

\begin{remark}

(1) If one needs to build a harmonic function explicitly for a given weighted Bratteli diagram $(B,c)$, then one can perform the following sequence of operations (we use the notation introduces above).

(i) Choose a vector $\ol f_1 \in \mathcal N_1$.

(ii) Check whether $\ol f_1 \in  \mbox{Col}(\overleftarrow{P}_1)$. If this is not the case, then there is no harmonic function $f = (f_i)$ such that $f_1 = \ol f_1$. If yes, find the solution set $\mathcal N_2$ of $\overleftarrow{P}_1 f_2 = \ol f_1$ and pick up a solution $\ol f_2$ from this set.

(iii)  For $\ol f_1$ and $\ol f_2$  determined in the previous steps, find a solution $\ol f_3$ of the equation $\overleftarrow{P}_2 f_3 = \ol f_2 - \overrightarrow{P}_1\ol f_1$  if it exists.

(iv) Repeat step (iii) for any $n$ and find  a vector $\ol f_{n+1}$ as a solution $\overleftarrow{P}_n f_{n+1} = \ol f_n - \overrightarrow{P}_{n-1}\ol f_{n-1}$. If this equation is consistent  for all $n$, then we detemine a harmonic function $\ol f = (\ol f_n)$ on $(B, c)$.

(2) For a weighted Bratteli diagram $(B, c)$, the property $\mbox{Rank}(\overleftarrow{P}_n) = |V_n|$ depends on both the structure of the Bratteli diagram, i.e., on the matrices $(A_n)$, and on the values of the function $c$. For instance, two columns of $A_n$ are identical if there are $x, y\in V_{n+1}$ such that the sets $s(r^{-1}(x))$ and $s(r^{-1}(y))$ coincide. But it is easy to make them linearly independent by an appropriate choice of the function $c$. Moreover, the property $\mbox{Rank}(\overleftarrow{P}_n) = |V_n|$ implicitly means that the sequence $(|V_n|)$ is not decreasing.

(3) As was mentioned in the algorithm above, if, for some $n >1$, the equation $\overleftarrow{P}_n f_{n+1} = g$ is inconsistent for any $g \in \mathcal G_n$, then the space $\mathcal Harm$ contains only constant harmonic functions defined on $V$.  This may happen only if $\mbox{Rank}(\overleftarrow{P}_n) < |V_n|$ for some level $V_n$. For a vertex $w\in V_{n+1}$, let $S(w) = s(r^{-1}(w))$ be the subset of vertices from $V_n$ that are connected to $w$ by an edge. Suppose that $S(w) = S(w')$ for $w, w' \in V_{n+1}$. If additionally $c(e) = c(e')$ where $s(e) = s(e')$, then the two columns of $\overleftarrow{P}_n$ coincide. For instance, it may happen for the simple random walk with $c=1$.

We consider now an example of weighted Bratteli diagram $(B,c)$ such that $|V_{i}| \leq |V_{i+1}|$ and $|V_{n+1}| < |V_n|$ where $n$ is the smallest number for which this inequality hold.  This means that the diagram looks like a ``bottleneck'' Bratteli diagram (see Figure 4).

\begin{figure}[ht!]
\begin{center}
\includegraphics[scale=0.8]{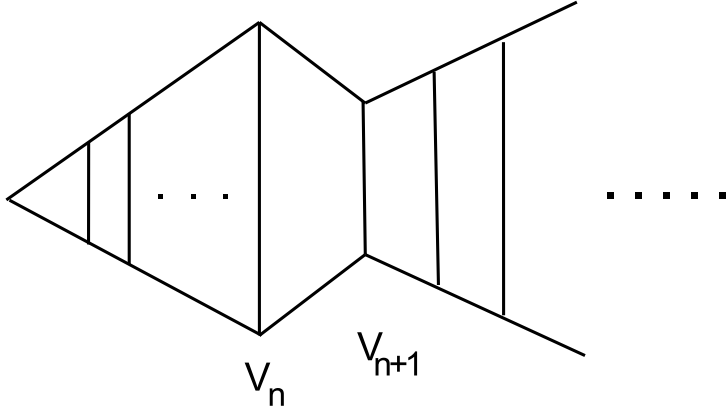}
\caption{Bottleneck Bratteli diagram}
\label{Bottleneck}
\end{center}
\end{figure}

We claim that in this case the space $\h arm$ of harmonic functions defined on $V$ is, in general, trivial. The reason for this is the simple observation that the number of linear equations in $\overleftarrow{P}_n f_{n+1} = g_n$ is bigger than that of unknowns. On the other hand, such a Bratteli diagram can have harmonic functions defined on a subset of $V$. Indeed, if the sequence $(|V_i|)$ is strictly increasing for all $i > n$, then there is a harmonic function $f$ which is nontrivially defined on $\bigcup_{i>n}V_i$. Thus, if a Bratteli diagram $B$ has infinitely many ``bottleneck'' levels, then $\h arm$ consists of constants only.

We also note that  the solution set  of the equation $\overleftarrow{P}_n f_{n+1} = g_n$ is always infinite provided $|V_n| < |V_{n+1}|$.

\end{remark}

The next interesting question is concerned the dimension of  the space of harmonic functions defined on a Bratteli diagram. We will show below that for a wide class of weighted Bratteli diagrams the dimension of $\h arm $ is infinite. On the other hand, there are examples of Bratteli diagrams for which the space of harmonic functions is nontrivial and finite-dimensional.

We refine the definition of a stationary Bratteli diagrams (see subsection \ref{subsect BD}): a weighted Bratteli diagram $(B, c)$ is called {\em stationary} if $A_n = A$ and $C_n = C$ for all $n \geq 1$.

It immediately follows from the  arguments we used in the algorithm construction that the following assertion holds.

\begin{theorem}\label{thm dimension of Harm}
(1) If a weighted Bratteli diagram $(B,c)$ is not of ``bottleneck'' type (that is  $|V_n| \leq |V_{n+1}|$ for every $n$), and, for infinitely many levels $n$, the strict inequality holds, then the space $\mathcal Harm$ is infinite-dimensional.

(2) There are stationary weighted Bratteli diagrams such that the space $\h arm$ is finitely dimensional.

\end{theorem}

\begin{proof} (Sketch) It follows from the fact that the cardinality of vertices at each level is an increasing sequence that the solution set $\mathcal N_n$ of the equation $\overleftarrow{P}_{n} f_{n+1} \in \mathcal G_{n}$ is infinite,  where we recall that $\mathcal G_{n} = \{ f_{n} - \overrightarrow{P}_{n-1} f_{n-1} \in \R^{|V_{n}|} :  f_{n} \in \mathcal N_{n},\ f_{n-1} \in \mathcal N_{n-1}\}$. Thus, we obtain the infinite-dimensional space of harmonic functions $f = (f_n)$, $f_n\in \mathcal N_n$.

The case of a stationary Bratteli diagram $B$ may lead to a finite-dimensional space $\h arm$. For instance, suppose that the equations
$$
\overleftarrow{P}_{n} f_{n+1} = f_n - \overrightarrow{P}_{n-1}f_{n-1}
$$
has a unique solution $f_{n+1}$ for every $n\geq 1$. This means that we have to impose some additional assumptions (which are rather obvious) on the sequence of matrices $\overleftarrow{P}_{n} $ that would guarantee the uniqueness of the solution. In this case, every harmonic function $f = (f_n)$ is completely determined by the vector $f_1$ that satisfies the equation $\overleftarrow{P}_0 f_1 = 0$.  Hence, $\dim (\h arm) = |V_1| - 1$. If the solution sets are not unique for infinitely many levels, then again the dimension of $\h arm$ is infinite.

\end{proof}

In Section \ref{sect HF on trees stationary BD}, we will return to stationary Bratteli diagrams and give an explicit formula for harmonic functions on a class of such diagrams.

\subsection{On the existence of monopoles and dipoles on a Bratteli diagram}

It turns out that the method of finding harmonic functions on a weighted Bratteli diagram $(B,c)$ works perfectly for another important classes of functions defined on $V$, namely, for {\em monopoles} and {\em dipoles}. We cite \cite{Jorgensen_Pearse2010, Jorgensen_Pearse2011, Dutkay_Jorgensen2010} for fundamentals about monopoles and dipoles in infinite networks.
We recall that, for a vertex $x\in V$, a function $w_x$ satisfying the equation $(\Delta w_x)(y) = \delta_x(y)$ is called a monopole. If additionally $w_x$ has  finite energy, then it defines an element of $\h_E$ also called a monopole.  It was mentioned in Remark \ref{rem dipole monopole} that if a monopole exists in $\h_E$ for some $x\in V$, then it exists for any vertex $z\in V$. Moreover, the existence of a monopole of finite energy on an electrical network $(G,c)$ is equivalent to the transience of this network.

We {\em claim} that all solutions of the equation $(\Delta w_x) = \delta_x(y)$ can be found according to the algorithm used for harmonic functions on $V$. We recall that a monopole $w_x$ can be treated as a harmonic function on the set $V \setminus \{x\}$. We are going to apply the algorithm described in subsection \ref{subsect algorithm} for monopoles and dipoles.

Suppose first that $x = o$. In order to determine a monopole $w_o$, we  solve the equation $\sum_{y\in V_1} c^{(0)}_{oy}f(y) = 1$ and find its solution set $\mathcal N'_1$. The other steps of this procedure are word-for-word repetition of the algorithm used for harmonic functions. If $x\in V_m, m \geq 1$, then to build a monopole $w_x$ we take the beginning of the algorithm to be the same as for harmonic functions. This means that the vectors $w_x(i) \in \mathbb R^{|V_i|}, i = 0,1, ... , m-1,$ can be found as  solutions of the sequence of equations
$$
\overleftarrow{P}_i w_x(i+1) = w_x(i) - \overrightarrow{P}_{i-1} w_x(i-1),
$$
where $w_x(i) = w_x|_{V_i}$. Here we assume  that these equations are consistent, otherwise we have only trivial solution. Thus, we can find $w_x(1), ... , w_x(m)$. The equation for $i =m$ is used to determine $w_x(m+1)$, and it looks slightly different
(it is written in vector form here):
$$
\overleftarrow{P}_m w_x(m+1) = w_x(m) - \overrightarrow{P}_{m-1} w_x(m-1) - \frac{1}{c(x)} \overline{\delta}_x
$$
where $\overline{\delta}_x$ is the vector $\{ \delta_x(y) : y\in V_m\}$. This relation follows from the equation $(\Delta w_x) = \delta_x(y)$,  $y \in V_m$. After that the procedure is the same as for harmonic functions.

In a similar manner, we can consider the set of {\em dipoles} on a Bratteli diagram $B = (V,E,c)$. We remind (see Remark \ref{rem dipole monopole}) that an element $v_x$ of $\h_E$ is called a dipole if it satisfies the equation $\Delta v_x = \delta_x - \delta_{o}$ where $x\in V$ and $o$ is the top of the diagram. It is known that dipoles always exist in $\h_E$ and can be found as follows \cite{Jorgensen_Pearse2010}. Let $f \in \h_E$ and $L_x f := f(x) - f(o)$, $x \in V$. Then $L_x$ is a bounded linear functional such that $|L_x f| \leq k \|f\|_{\h_E}$. Then $v_x$ is a unique element of $\h_E$ such that for any $f \in \h_E$
$$
\langle v_x, f\rangle = f(x) - f(o).
$$

One can again use the algorithm given for harmonic functions and apply it for determining dipoles as functions on $V$. Suppose $x\in V_m$. Since $v_x$ can be written as a sequence of vectors $(v_x(i))$, and $v_x$  must satisfy the equation
$$
(\Delta v_x)(y) = \begin{cases} 1, \ & y = x\\
                                                 -1, \ &  y = o\\
                                                 0, \ & y \in V \setminus \{o, x\},
\end{cases}
$$
we can solve consequently the equations $\sum_{y\in V_1} c^{(0)}_{oy}v_x(1)(y) = -1$,
$\overleftarrow{P}_i v_x(i+1) = v_x(i) - \overrightarrow{P}_{i-1} v_x(i-1)$ (for $i =1, ..., m-1$),  and finally $\overleftarrow{P}_m v_x(m+1) = v_x(m) - \overrightarrow{P}_{m-1} v_x(m-1) - \frac{1}{c(x)} \overline{\delta}_x$. Then we again proceed as in the case of harmonic functions.


\section{Harmonic functions through Poisson kernel}\label{sect HF Poisson kernel}

\subsection{Integral representation of harmonic functions}
In this subsection we will assume that the considered networks are transient; see Definition \ref{def recurrence and transience}. Motivated by the paper \cite{Ancona_Lyons_Peres1999}, we are going to find an integral representation of harmonic functions in terms of a Poisson kernel, and investigate the convergence of harmonic functions on the path space of a Bratteli diagram.

Let $(B, c)$ be  a weighted Bratteli diagram where $B = (V,E)$ and the conductance function $c$ is chosen so that the network $(B,c)$ is transient. The transition probabilities matrix $P = (p_{xy} : x,y \in V)$ defines a random walk on the set of all vertices $V$. Let $\Omega \subset V^{\infty}$ be the set of all paths $\omega = (x_0, x_1, ..., x_n,...)$ where $(x_{i-1}x_i) \in E$. For a fixed $x \in V$,  we denote by $\Omega_x$ the subset of $\Omega$ formed by those paths that starts with $x$. Then  $\mathbb P_x$ denotes the Markov measure on $\Omega_x$ generated by $P$ (see Section \ref{sect_Basics} for details).

Let $X_i : \Omega_x \to V$ be the random variable on $(\Omega_x, \mathbb P_x)$ such that $X_i(\omega) = x_i$. For a given vertex $x\in V$ and some level $V_n \subset V$ such that $x \notin V_n$, we determine the function of stopping time (more information on this notion can be found, for instance, in \cite{Du2012, Sokol2013}):
$$
\tau(V_n) (\omega) = \min\{i \in \N : X_i(\omega) \in V_n\}, \ \ \omega\in \Omega_x.
$$
For $x \in V_n$, we set $\tau(V_n)(\omega) = 0$. The value $\tau(V_n) (\omega)$ shows when the orbit $\omega$ reaches $V_n$ at the first time.

Assuming that the random walk $(X_i)$ defined by $P$ on $(B,c)$ is transient, we observe that  $\tau(V_n)(\omega) $ satisfies the following property.

\begin{lemma}\label{lem_stopping_time}
Let $(B,c)$ be a transient network, and $W_{n -1} = \bigcup_{i=0}^{n-1} V_i$. Then for every $n \in \N$ and any $x\in W_{n-1}$, there exists $m > n$ such that for $\mathbb P_x$-a.e. $\omega\in \Omega_x$
\be
\tau(V_{i+1})(\omega) = \tau(V_{i})(\omega) +1, \ \ i \geq m.
\ee
\end{lemma}

\begin{proof}
The result immediately follows from relation  (\ref{eq_transience}) of Definition \ref{def recurrence and transience}.
\end{proof}

We recall that any  real-valued function $ f $ on $V = \bigcup_n V_n$ is identified with a sequence of vectors $(f_n)$ where $f_n = f|_{V_n}$.

Now we fix a vector $f_n \in \R^{|V_n|}$ and define the function $h_n : X \to \R$ by setting
\be\label{eq for h_n}
h_n(x) := \mathbb E_x( f_n \circ X_{\tau(V_n)}) = \int_{\Omega_x} f_n(X_{\tau(V_n)}(\omega)) d\mathbb P_x(\omega), \ \ n\in \N.
\ee

\begin{lemma}\label{lem_harm_fn_h_n} For a given function $f = (f_n)$, and, for  every $n$, the function $h_n(x)$ is harmonic on $V \setminus V_n$ and $h_n(x) = f_n(x), x \in V_n$. Furthermore, $h_n(x)$ is uniquely defined on $W_{n-1}$.
\end{lemma}

\begin{proof}
It is easy to see from the definition of $h_n(x)$ that $h_n(x) = f_n(x)$ when $x \in V_n$ because in the relation $h_n(x) = \mathbb E_x( f_n \circ X_{\tau(V_n)}(\omega)) $ the right side does not depend on $\omega$ and $\tau(V_n) =0$.

In order to show that $h_n$ is harmonic on $V \setminus V_n$, we fix arbitrary $x \in V \setminus V_n$ and compute
\begin{eqnarray*}
h_n(x)  &=& \sum_{y\sim x} p(x, y) \mathbb E_x(f_n \circ X_{\tau(V_n)} | X_1 = y)\\
   &=&  \sum_{y\sim x} p(x, y) \mathbb E_y(f_n \circ X_{\tau(V_n)} )\ \ \mbox{(using\ the\ Markov\ property}) \\
   &=&  \sum_{y\sim x} p(x, y) h_n(y)\\
   & = & (Ph_n)(x)
\end{eqnarray*}

The fact that $h_n(x)$ is uniquely determined on $W_{n-1}$ follows from the uniqueness of the solution of the Dirichlet problem
$$
(\Delta u)(x) = 0, \ x \in W_{n-1}, \ \ \mbox{and} \ \ u(x) = f_n(x),\ x \in V_n
$$
where  $V_n = \partial W_{n-1}$ (see Section \ref{sect_Basics} where the Dirichlet problem is discussed).
\end{proof}

Our main result is based on the proved lemmas, and it is formulated in the following theorem.

\begin{theorem}\label{thm convergence of h_n}
Let $f = (f_n) \geq 0$ be a function on $V$ such that $\overleftarrow{P}_n f_{n+1} = f_n$. Then the sequence $(h_n(x))$ defined in (\ref{eq for h_n}) converges pointwise to a harmonic function $H(x)$. Moreover, for every $x\in V$, there exists $n(x)$ such that $h_i(x) = H(x), i \geq n(x)$. Equivalently, the sequence $(f_n\circ X_{\tau(V_n)})$ converges in $L^1(\Omega_x, \mathbb P_x)$.
\end{theorem}

\begin{proof} Fix some vertex $x$ in $V$, suppose $x \in V_{\ell}$. By Lemma \ref{lem_stopping_time}, we can find $n > \ell$ such that $\tau(V_{n+1})(\omega)
= \tau(V_{n})(\omega) +1$ for a.e. $\omega\in \Omega_x$. Then
\begin{eqnarray*}
  h_{n+1}(x)   & = &\mathbb E_x( f_{n+1} \circ X_{\tau(V_{n})+1})\\
   &=& \sum_{z\in V_{n+1}} f_{n+1}(z) \mathbb P_x(X_{\tau(V_{n})+1} = z \ |\ X_0 =x) \\
  &=& \sum_{y\in V_n} \sum_{z\in V_{n+1}} f_{n+1}(z) \overleftarrow{p}^{(n)}(y,z) \mathbb P_x(X_{\tau(V_{n})} = y \ |\ X_0 =x)  \\
   &=& \sum_{y\in V_n}  f_n(y)  \mathbb P_x(X_{\tau(V_{n})} = y\ |\ X_0 =x) \\
   & = & h_n(x).
\end{eqnarray*}
We proved that for any $x\in V$ the sequence $(h_n(x))$ stabilizes eventually, and we can  set $H(x) = \lim_n h_n(x)$. Obviously, $H$ is harmonic.

Let $\psi(\omega)$ be a function from $L^1(\Omega_x, \mathbb P_x)$ such that
$\|f_n\circ X_{\tau(V_{n})} - \psi \|_{L^1} \to 0$ as $n \to \infty$. It follows from the proved results that $H(x) = \mathbb E_x(\psi)$.

\end{proof}

\begin{remark} 1. We notice that the condition $\overleftarrow{P}_n f_{n+1} = f_n$ need not to be true for all $n$. It suffices to have this property for all sufficiently large $n$; the function $f_i$ can be chosen arbitrary for a finite set of $i$'s.

2. In \cite{Ancona_Lyons_Peres1999}, the following  statement was proved: If a reversible Markov chain $X_n$ is transient and $f$ is a (harmonic) function of finite energy, then $(f\circ X_n)$ converges almost everywhere. Our result above is of the same nature, but we do not require that the harmonic function has finite energy.

\end{remark}

\subsection{Properties of harmonic functions on a Bratteli diagram}

Given a function $f : V \to \R$, define the {\em current}  $I(x)$ through $x\in V$ as
$$
I(x) := \sum_{y\sim x} c_{xy}(f(x) - f (y)).
$$

The following statement represents a form of the Kirchhoff law and can serve as a  characterization of harmonic functions defined on vertices of a Bratteli diagram.

\begin{lemma}\label{current for harm fns}
 A function $f : V \to \R$ is  harmonic on a weighted Bratteli diagram $(B,c)$ if and only if for every $x \in V_n, n \geq 1$,
$$
I_{in}(x) := \sum_{y\in V_{n-1}} c_{xy}(f(x) - f (y)) = \sum_{z\in V_{n+1}} c_{xz}(f(z) - f(x)) =: I_{out}(x).
$$
Hence,  the incoming current is equal to outgoing current for every vertex if and only if the function $f$ is harmonic.

\end{lemma}

Based on this result,  we can define, for $x\in V_n$,
$$
I_n(x) := I_{in}(x), \ \ \ \mbox{and} \ \ \ I_n = \sum_{x\in V_n} I_n(x).
$$

\begin{lemma}\label{lem I_n = const}
Let $f$ be a harmonic function on a weighted Bratteli diagram $(B,c)$. Then,
for any $n\geq 1$,
$$
I_n = I_1,
$$
\be\label{sum I_n^2}
\sum_{x\in V_n} (I_n(x))^2 \geq \frac{I_1^2}{|V_n|},
\ee
where $I_1 = \sum_{x\in V_{1}} c_{ox}(f(x) - f (o))$.
\end{lemma}

\begin{proof}
To prove the first relation it suffices to show that $I_n = I_{n+1}$. Indeed, if we denote $f_n = f|_{V_n}$, then we use Lemma \ref{current for harm fns} to obtain
\begin{eqnarray*}
I_{n+1} &= & \sum_{z\in V_{n+1}} \sum_{x\in V_n} c_{xz} (f_{n+1}(z) - f_n(x)) \\
  &= &\sum_{x\in V_n} \sum_{z\in V_{n+1}} c_{xz} (f_{n+1}(z) - f_n(x))  \\
   &=&  \sum_{x\in V_n} \sum_{y\in V_{n-1}} c_{xy} (f_{n}(x) - f_{n-1}(y)) \\
 & = & I_n
\end{eqnarray*}

We apply the Schwarz' inequality to the relation $\sum_{x\in V_n} I_n(x) = I_1$:
\begin{eqnarray*}
I_1^2 &=& \left(\sum_{x\in V_n} I_n(x)\right)^2 \\
   & \leq & (\sum_{x\in V_n} 1) \cdot \sum_{x\in V_n} (I_n(x))^2 \\
 &=& |V_n| \sum_{x\in V_n} (I_n(x))^2.
\end{eqnarray*}
Thus, (\ref{sum I_n^2}) is proved.
\end{proof}

We formulate the following statement for harmonic functions only although  it can be given in more general terms of subharmonic functions (that is $(\Delta f)(x) \le 0$ for every $x$) and superharmonic functions (that is $(\Delta f)(x) \ge 0$ for every $x$) functions which are not discussed here.

\begin{proposition}\label{prop max/min principl}
Let $(B, c)$ be a weighted Bratteli diagram and $G_n = \{o\} \cup V_1 \cup \cdots \cup V_n$. Then for any nontrivial harmonic function $f : V \to \R$
$$
\max \{f(x) : x \in G_n\} = \max \{f(x) : x \in \partial G_n = V_{n}\} =: M_{n}(f).
$$
$$
\min \{f(x) : x \in G_n\} = \min \{f(x) : x \in \partial G_n = V_{n}\} =: m_{n}(f).
$$
Moreover, for any $x, y \in G_n$,
\be\label{inequality for harm max - min}
f(x) - f(y) \leq M_{n}(f) - m_{n}(f),\ \ n\in \N.
\ee
The sequence $\{M_n(f)\}$ is strictly increasing, and the sequence $\{m_n(f)\}$ is strictly decreasing.
\end{proposition}

\begin{proof}
The fact that a harmonic function assumes its maximum and minimum values at boundary points is the well known property of harmonic functions. In other words, this principle states that if a harmonic function attains its maximum/minimum at an inner vertex of a connected subgraph then the function must be constant.  Relation (\ref{inequality for harm max - min}) is then obvious.

To show the last statement, we check that $M_n(f) < M_{n+1}(f)$ for every $n \geq 1$. Suppose this is not true. If it were $M_k(f) > M_{k+1}(f)$ for some $k$, then $f$ would be a constant since it takes its maximum at an inner vertex of the connected graph $G_{k+1}$. Assume now that  $M_k(f)  = M_{k+1}(f)$ for some $k$. Let $f(x) = M_k(f)$ for $x\in V_k $, and $f(z) = M_{k+1}(f) = M_k(f)$ for $z \in V_{k+1}$. Let $y \sim x$ be a vertex from $V_{k+1}$ (it may be that $x$ and $z$ are not adjacent). Because $f(x)$ is maximal for $G_{k+1}$, then $f(y) = f(x) = M_{k+1}$, and we conclude that $f$ must be a constant function. This is a contradiction.

The sequence $\{m_n(f)\}$ is considered similarly.
\end{proof}

It can be noticed that in conditions of Proposition \ref{prop max/min principl} one can always assert that the sequence $\{M_n(f)\}$ is formed by positive numbers and the sequence   $\{m_n(f)\}$ has only negative terms  provided $f(o) = 0$.

\begin{corollary}
Let $(B(V,E), c)$ be a weighted Bratteli diagram. A harmonic function $f :V \to \R$ belongs to $\ell^\infty(V)$ if and only if the sequences $\{M_n(f)\}$,  $\{m_n(f)\}$ have finite limits.
\end{corollary}

\section{Harmonic functions on trees, the Pascal graph, and stationary Bratteli diagrams}\label{sect HF on trees stationary BD}

In this section, we focus on some particular cases of Bratteli diagrams. They are trees, the Pascal graph,  and stationary Bratteli diagrams. In the latter case, the incidence matrix $A_n$ does not depend on level $n$.

\subsection{Harmonic functions on trees} Our goal is twofold: we first show that the algorithm of finding harmonic functions and monopoles/dipoles can readily be  applied to weighted trees. Working with a tree, we will use the notation introduced for Bratteli diagrams.

\begin{proposition}\label{prop HF on tree}
Let $T$ be a tree with conductance function $c$. The space $\h arm$ of harmonic functions on the electrical network $(T,c)$ is infinite-dimensional. Any harmonic function can be found by the algorithm given in Proposition \ref{prop suff cond for exist harm fns}.
\end{proposition}

\begin{proof}
It is obvious that for any weighted tree $(T, c)$ the matrix $\overleftarrow{P}_n$, which maps $\mathbb R^{|V_{n+1}|}$ to $\mathbb R^{|V_{n}|}$, has the linearly independent columns, and therefore $\mathrm{Rank}(\overleftarrow{P}_n) = |V_n|$.
This means that Proposition \ref{prop suff cond for exist harm fns} holds for any tree. The equation $\overleftarrow{P}_n f_{n+1} = g$ has infinitely many solutions for every $n$. In fact, every row in $\overleftarrow{P}_n $ corresponds an  equation in the system $\overleftarrow{P}_n f_{n+1} = g$, and this equation is solved independently of the other rows because of the tree structure (every column of $\overleftarrow{P}_n$ has exactly one non-zero entry). This shows that $\h arm$ is infinite-dimensional. More precisely, in order to find a harmonic function on the tree $T$, one needs to solve the equation with respect to $f_{n+1}(z)$, for every $x\in V_n$,
\be\label{eq for tree}
f_{n}(y) - \overrightarrow{p}^{(n-1)}_{xy} f_{n-1}(x) = \sum_{z\sim x} \overleftarrow{p}^{(n)}_{xz} f_{n+1}(z),
\ee
where $y \in V_{n-1}$ is uniquely determined by $x$. Hence, the solution set of (\ref{eq for tree}) has dimension $d_x -1$ where $d_x$ is the number of successors of $x$ (see Figure 5).

\end{proof}

\begin{figure}[ht!]
\begin{center}
\includegraphics[scale=0.85]{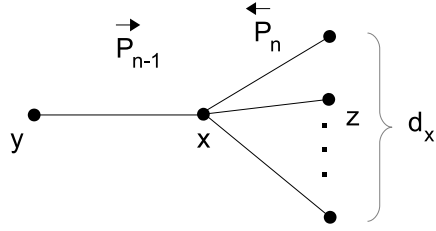}
\caption{Part of the tree corresponding to (\ref{eq for tree})}
\label{Fig_tree}
\end{center}
\end{figure}

\begin{example}[Symmetric harmonic functions on the binary tree]\label{ex hf on a tree}
In this example, we give explicit formulas for a class of harmonic functions defined on the binary tree. We regard  a homogeneous tree $T$ as a special case of a Bratteli diagram, and we keep the same notation as in Section \ref{sect HF on BD}.

Let $x_0$ be the root of the binary tree $T$, and let $V_n$ denote the set of vertices on the distance $n$ from the root. Next, we assume that the conductance function $c = c(e), e \in E$, has the property: $c(e) = \lambda^n$ for all $e \in E_n, n\geq 0$. Hence, the associated matrices $C_n$ are of the size $2^n \times 2^{n+1}$, and the $i$-th row of $C_n$ consists of all zeros but $c^{(n)}_{i, 2i-1} = c^{(n)}_{i, 2i} =\lambda^n$.
Denote by $x_n(1), ... , x_n(2^n)$ the vertices of $V_n$ enumerated from the top to the bottom, see Figure 6.

\begin{proposition}\label{prop formula for HF on tree} Let $(T, c)$ be the weighted binary tree defined above.
For each positive $\lambda$ there exists a unique harmonic function $f = f_\lambda$ satisfying the following conditions:

(1) $f(x_0) =0$;

(2) $f(x_1(1)) = - f(x_1(2)) = \lambda$ and
$$
f(x_n(1)) = - f(x_n(2^n)) = \frac{1 + \cdots + \lambda^{n-1}}{\lambda^{n-2}},\ n\geq 2;
$$

(3) function  $f$ is  constant on each of subtrees $T_i$  and $T_i'$ whose all infinite paths start at the roots  $x_i(1)$ and $x_i(2^i)$, respectively, and go through the vertices $x_{i+1}(2)$ and $x_{i+1}(2^{i+1} -1)$, $i \geq 1$ (see Figure 6).

\end{proposition}

\begin{figure}[ht!]
\begin{center}
\includegraphics[scale=0.85]{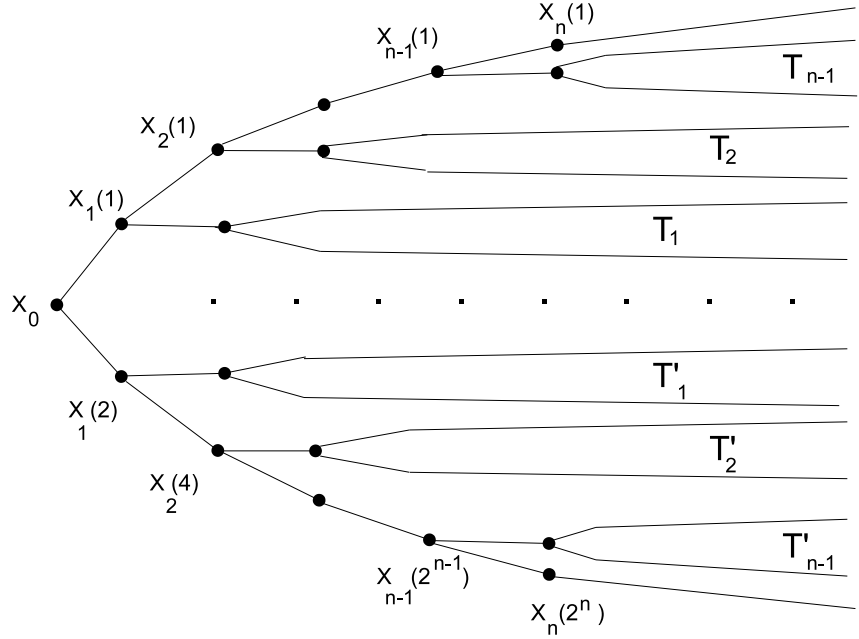}
\caption{Symmetric harmonic function on the binary tree}
\label{HF_on_tree}
\end{center}
\end{figure}

It is natural to call $f_\lambda$ the symmetric harmonic functions on the binary tree defined by $\lambda$. It should be clear that the same construction can be applied to any homogeneous tree.

\begin{proof}
We prove the statement by induction. It is straightforward to verify that the function $f$ is harmonic at $x_1(1)$ because $f(x_0) = 0, f(x_1(1)) = \lambda, f_2(x_2(1)) = 1 +\lambda$,  $f(x_1(1)) = f(x_2(2)) = \lambda$, and $c_{x_0x_1(1)}  =  1, c_{x_1(1)x_2(1)} =c_{x_1(1)x_2(2)} =\lambda$. We notice that the function $f$ is harmonic at $x_1(1)$ and takes the same value at $x_2(2)$. Therefore, by the maximum/minimum principle, it must be constant on the connected subgraph $T_1$.

Suppose that the proposition is proved for $ i = 1, ... , n-1$. We have to show that $f$ is harmonic at $x_n(1)$. In other words, $f$ must satisfy the relation
\be\label{eq hf on tree}
(2\lambda^n + \lambda^{n-1})f(x_n(1)) = \lambda^n f(x_{n+1}(1)) + \lambda^n f(x_{n+1}(2)) + \lambda^{n-1} f(x_{n-1}(1)).
\ee
But  (\ref{eq hf on tree}) holds if and only if
\be\label{eq f_lambda}
f(x_{n+1}(1)) = \frac{1 + \cdots + \lambda^{n}}{\lambda^{n -1}}.
\ee

The facts that $f$ is uniquely determined by conditions (1),  (2), and $f$ is constant on subtrees $T_i$ and $T'_i$ are now obvious.
\end{proof}

In Section \ref{sect HF and energy}, we will show that the found harmonic function $f$ on the binary tree has finite energy if and only if $\lambda > 1$.
\end{example}

If $\lambda =1$, then it is interesting to observe that the values of $f_1$ on $\{x_i(1)\}$ are natural  numbers, $f_1(x_i(1)) = i$. Moreover, $f_1 =  i$ everywhere on the subtree $T_i$.

\subsection{Harmonic functions on the Pascal graph}

We will discuss in this subsection the existence of harmonic functions on the {\em Pascal graph} and give an explicit formula for such a  harmonic function in the simplest case when the conductance function $c =1$.
By definition, the Pascal graph is the 0-1 Bratteli diagram  with the sequence of incidence matrices $(A_n)_{n\ge 0}$ of the size $(n+1) \times (n+2)$ where
$$
A_n =  \left(
  \begin{array}{ccccccc}
    1 & 1 & 0 & 0 & \cdots & 0 & 0\\
    0 & 1 & 1 & 0 & \cdots  & 0 & 0 \\
    0 & 0 & 1 & 1 & \cdots & 0  & 0\\
    \cdots & \cdots & \cdots & \cdots  & \cdots  & \cdots  & \cdots\\
    0 & 0 & 0 & 0& \cdots  & 1 & 1 \\
  \end{array}
\right)
$$
In other words, the nonnegative part of the lattice $\Z^2$ is represented as the Pascal graph, where $(0,0)$ is the top vertex of the diagram, and the levels $V_n$ are formed by the vertices $(x,y)$ from $\Z^2_+$ such that $x + y =n$, $n \geq 1$.

Every vertex $v\in V_n$ of the Pascal graph can be enumerated by two numbers (coordinates) $(n, i)$, where  $0 \leq i \leq n$ is the position of $v$ in $V_n$ (it is assumed that the set of vertices $\{(x, 0) : x \in \N_0\}$ is the upper bound line of the graph) (see Figure 7 where the meaning of the assigned numbers will be explained below).

We first {\em claim} that the algorithm of finding harmonic functions is applicable to the Pascal graph. This follows from the fact that, for any choice of the conductance function $c$, the rank of $\overleftarrow{P}_n$ equals $n+1$. More precisely, we claim that the equation
\be\label{eq HF for Pascal}
\overleftarrow{P}_n f_{n+1} = f_n - \overrightarrow{P}_{n-1} f_{n-1}
\ee
always has a solution for $f_{n+1}$ assuming that $f_n$ and $f_{n-1}$ have been determined in the previous steps. Moreover, the solution set of this equation is one-dimensional for every $n$. For instance, if we are looking for a harmonic function $f$ on $V$ satisfying $f(0,0) =0$, then $f(1, 0) = - f(1, 1)$.

Equation (\ref{eq HF for Pascal}) becomes more transparent if we additionally require that the
conductance function $c$ is defined by the rule $c(e) = \lambda^n$, for any $e \in E_n$, and the harmonic function $f$ vanishes at $(0,0)$. Then  one can easily find the explicit form of  $\overleftarrow{P}_n$ for any $n \geq 1$:
$$
\overleftarrow{P}_n =
\left(
  \begin{array}{cccccc}
    \frac{\lambda}{1 + \lambda} & \frac{\lambda}{1 + \lambda}  & 0 & 0 & \cdots  & 0\\
    0 & \frac{\lambda}{2 + \lambda}  & \frac{\lambda}{2 + \lambda} & 0 & \cdots & 0\\
    0 & 0 & \frac{\lambda}{2 + \lambda} & \frac{\lambda}{2 + \lambda} & \cdots & 0\\
    \cdots & \cdots & \cdots & \cdots & \cdots & \cdots \\
    0 & 0 & 0 & \cdots &  \frac{\lambda}{1 + \lambda}  &  \frac{\lambda}{1 + \lambda}  \\
  \end{array}
\right)
$$

We summarize the above observations in the following statement. We say that a harmonic function $h$ on the Pascal graph $(B, c)$ is {\em symmetric} if it satisfies  the condition $h(n, i) = - h(n, n-i)$ for any $n$ and $0 \leq i \leq n$.

\begin{lemma} Let $(B,c)$ be a weighted Pascal graph. Then  $\h arm$ is a non-empty infinite dimensional space containing the subspace of symmetric harmonic functions. Moreover, this subspace is also infinite dimensional.
\end{lemma}

We consider here the electrical network $(B, 1)$ defined on the Pascal graph with conductance function $c =1$. For such $c$, we can find an explicit example of symmetric harmonic function.

\begin{proposition}\label{prop HF on Pascal} Define $h(0,0) = 0$ and  set, for every vertex $v = (n, i)$,
\be\label{HF on Pascal}
h(n,i) := \frac{n(n+1)}{2} - i(n+1),
\ee
 where  $0 \leq i \leq n$ and $n \geq 1$.
Then $h : V \to \mathbb R$ is an integer-valued harmonic function on $(B, 1)$ satisfying the symmetry condition $h(n, i) = - h(n, n-i)$ (see Figure 7) .
\end{proposition}

\begin{proof}
In order to prove the proposition, it suffices to check that the following properties hold for any $(n, i)$:
$$
3h(n, 0) = h(n-1,0) + h(n+1, 0) + h(n+1, 1),
$$
in case when there are three neighbors for $v = (n,0)$, and
$$
4h(n, i) = h( n-1, i - 1) +h(n-1, i) + h(n+1, i) + h(n+1, i+1),
$$
in case when there are four neighbors for $v = (n,i)$. These relations are proved by direct computation based on (\ref{HF on Pascal}).
\end{proof}

\begin{figure}[ht!]
\begin{center}
\includegraphics[scale=0.75]{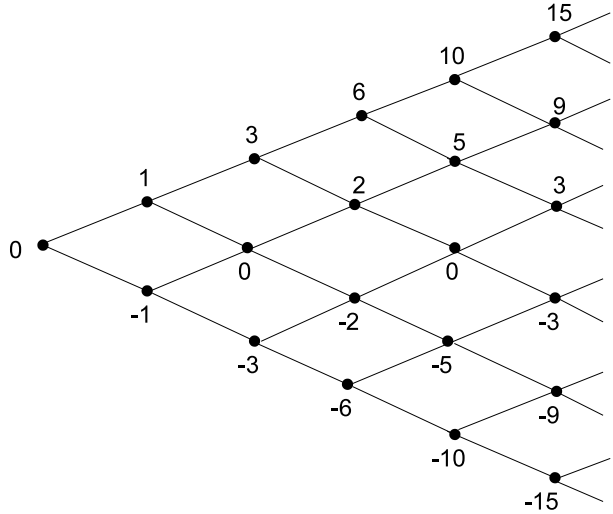}
\caption{Harmonic function on the Pascal graph.}
\label{Fig_Pascal}
\end{center}
\end{figure}

{\em Question}. It would be interesting to find explicit formulas for harmonic function on the Pascal graph $(G, c)$ with some non-trivial conductance $c$, in particular, for $c(e) = \lambda^n, e \in E_n$.

\subsection{Harmonic functions on stationary Bratteli diagrams}

We recall that for any weighted graph $(G, c)$ the question about the existence of nontrivial harmonic function can be completely answered in two following cases: (1) if the network $(G,c)$ is recurrent, then $\h arm$ consists of constant harmonic functions only, and (2) if $\sum_{e \in E} c(e) < \infty$, then again $\mathcal Harm$ is trivial \cite{Georgakopoulos2010}.

In case of stationary Bratteli diagrams, we can clarify the structure of $\h arm$.
Let $A$ be the incidence matrix of a stationary Bratteli diagram $B$, and suppose  $A$ has $d \times d$ size. Assume that the conductance function $c$ has the property: $c_e = c_{xy} = \lambda^n$ for any $e \in E_n$ with $s(e) =x \in V_n, r(e)= y \in V_{n+1}$. Then the associated matrix $C_n = \lambda^n A$ ($C_n$ is defined in Section \ref{sect HF on BD}).

We start with rewriting relation (\ref{C_n relations with x for harm f-n}) according to the made assumptions about $B$ and $c$. It has now the form
\be\label{eq stationary}
A^T(f_{n-1} - f_n(x) \mathbf 1_d) + \lambda A(f_{n+1} - f_n(x) \mathbf 1_d) =0
\ee
where $x \in V_n, n \in \N$ and $\mathbf 1_d = (1, ... , 1)^T \in \R^d$.

\begin{proposition}\label{prop stat BD}
Let the weighted Bratteli diagram $(B,c)$ be as defined above. Suppose that $A = A^T$ and $A$ is invertible. Then any harmonic function $f =(f_n)$ on $(B,c)$ can be found by the formula:
\be\label{eq HF stat BD}
f_{n+1} (x) = f_1(x) \sum_{i =0}^n \lambda^{-i}
\ee
where  $x \in V$.
\end{proposition}

\begin{proof}
We first observe that, as $V_n = V$ for $n \geq 1$, we can interpret  relation (\ref{eq stationary}) as a sequence of relations between vectors $f_n$ that hold on the same space $\R^d$, where $x$ is  any vertex from $V$. Moreover, using the properties of  $A$, we obtain
$$
(f_{n-1} - f_n(x) \mathbf 1_d) + \lambda (f_{n+1} - f_n(x) \mathbf 1_d) =0.
$$
From this equation between vectors, we deduce that it holds for any coordinate $y\in V$, in particular, for $y =x$. Hence,
$$
\lambda (f_{n+1}(x) - f_n(x)) = f_{n}(x) - f_{n-1}(x),
$$
and, for every $n \geq 1$,
$$
f_{n+1}(x)  - f_n(x) = \frac{1}{\lambda^n}(f_1(x) - f_0(x)) = \frac{1}{\lambda^n}f_1(x).
$$
Summation of these relations gives
$$
f_{n+1} (x) = f_1(x) \sum_{i =0}^n \lambda^{-i}.
$$
\end{proof}

It follows from Proposition \ref{prop stat BD} that we can explicitly describe elements of the space $\h arm$ for this class of weighted stationary Bratteli diagrams.

\begin{corollary} Let $(B, c)$ be as in Proposition \ref{prop stat BD}.

(1) The dimension of the space $\h arm$  is $d-1$ where $d = |V|$.

(2) If $\lambda > 1$, then every harmonic function on $(B, c)$ is bounded.

\end{corollary}
\begin{proof} The proof follows from relation (\ref{eq HF stat BD}).

\end{proof}

\section{Harmonic functions of finite and infinite energy}\label{sect HF and energy}

In this section we discuss some results about the energy of harmonic functions.
Recall that, given a function $u : V \to \R$ defined on the vertex set $V$ of an electrical network $(G, c)$, its energy $\|u\|_{\mathcal H_E}$ is computed by
(\ref{norm in H_E}). In case of a harmonic function $h \in \h arm$, one can also use  the formulas from Lemma \ref{lem for energy of harm fns} to find the energy of $h$.

Let $(B, c)$ be a weighted Bratteli diagram. Denote
\be\label{max/min for c}
\beta_n = \max\{c(x) : x \in V_n\}.
\ee

\begin{theorem}\label{thm infinite energy}
Let $f$ be a harmonic function on a weighted Bratteli diagram $(B, c)$.
Then
\be\label{energy estimation from below}
\sum_{n =0}^\infty \frac{I_1^2}{\beta_n |V_n|} \leq   \|f\|_{\mathcal H_E}^2,
\ee
where $I_1= \sum_{x\in V_{1}} c_{ox}(f(x) - f (o))$ was defined in Lemma \ref{lem I_n = const}.
\end{theorem}

\begin{proof} Let $ f : V \to \R$ be a function on a vertex set of $(B,c)$.
The norm $\|f\| = \|f\|_{\mathcal H_E}$  in the Hilbert space $\mathcal H_E$ is computed  by the formula
\begin{eqnarray*}
  \|f\|^2 &=& \frac{1}{2}\sum_{x,y \in V} c_{xy}(f(x) - f(y))^2  \\
   &=& \frac{1}{2}\sum_{n\in \N} \sum_{x\in V_n} \left[ \sum_{y\in V_{n-1}}
 c_{xy}(f(x) - f(y))^2  + \sum_{z\in V_{n+1}} c_{xz}(f(x) - f(z))^2  \right]\\
  \end{eqnarray*}
If $f$ is harmonic, then we can use Lemma \ref{current for harm fns}  and write
\begin{eqnarray*}
(I_n(x))^2  &= & \left(\sum_{y\in V_{n-1}} \sqrt{c_{xy}} (  \sqrt{c_{xy}} (f(x) - f(y))\right)^2 \ \ \ \ \ \ \ \mbox{(by\ the\ Schwarz'\ inequality)}\\
    &\leq &  \left( \sum_{y\in V_{n-1}} c_{xy} \right)  \left( \sum_{y\in V_{n-1}} c_{xy}  (f(x) - f(y))^2\right).
\end{eqnarray*}
Similarly, we can use the formula for outgoing current (Lemma \ref{current for harm fns}) and conclude that
$$
(I_n(x))^2 \leq \left( \sum_{z\in V_{n+1}} c_{xz} \right)  \left( \sum_{z\in V_{n+1}} c_{xz}  (f(x) - f(z))^2\right).
$$
Adding the last two inequalities, we obtain
\be\label{sum I_n(x)^2 leq}
2(I_n(x))^2 \leq c(x) \sum_{y \sim x} c_{xy}  (f(x) - f(y))^2.
\ee
It follows from (\ref{sum I_n^2}), (\ref{max/min for c}), and (\ref{sum I_n(x)^2 leq}) that
\begin{eqnarray*}
  \frac{I_1^2}{|V_n|}  & \leq &  \sum_{x \in V_n}  (I_n(x))^2\\
   & \leq & \frac{1}{2} \sum_{x\in V_n} c(x) \sum_{y \sim x} c_{xy}  (f(x) - f(y))^2 \\
   &=& \frac{1}{2} \beta_n \sum_{x\in V_n} \sum_{y \sim x} c_{xy}  (f(x) - f(y))^2. \\
  \end{eqnarray*}
Hence,
$$
\sum_{n=0}^\infty \frac{I_1^2}{\beta_n |V_n|} \leq \frac{1}{2} \sum_{n=0}^\infty \sum_{x\in V_n} \sum_{y \sim x} c_{xy}  (f(x) - f(y))^2 = \|f\|^2,
$$
and the proposition is proved.
\end{proof}

It immediately follows from the proved inequality (see (\ref{energy estimation from below})) that the following result holds.

\begin{corollary}\label{cor harm functions have inf energy}
Suppose that a weighted Bratteli diagram $(B,c)$ satisfies the condition
 \be\label{divergent series}
\sum_{n =0}^\infty (\beta_n |V_n|)^{-1} = \infty
\ee
where $V = \bigcup_n V_n$ and $\beta_n = \max\{c(x) : x \in V_n\}$.
Then any nontrivial harmonic function has infinite energy, i.e.,
$$
\mathcal Harm \cap \mathcal H_E = \{\mathrm{const}\}.
$$
In other words, such a $(B,c)$ does not support nonconstant harmonic functions of finite energy.
\end{corollary}

We illustrate this result by considering the following examples. We consider the harmonic functions found in Section \ref{sect HF on trees stationary BD} for the binary tree, the Pascal graph, and a stationary Bratteli diagram. In each of these cases, we compute the energy norm of harmonic functions.

\begin{example} [Binary tree] Let $T$ be the binary tree, and the conductance function $c$ is defined by the relation $c(e) = \lambda^n$ for all $e \in E_n, n\in \N_0$. We can treat $(T,c)$ as a special class of Bratteli diagrams, and we apply the notation used for Bratteli diagrams. Let $V_n = \{x_n(1), .... , x_n(2^n)\}$ be the vertices of the $n$-th level. We recall that a symmetric harmonic harmonic function $f_\lambda =(f_n)$ was found in Proposition \ref{prop formula for HF on tree}.

\begin{lemma} Fix $\lambda$, then each of the harmonic functions $f_\lambda$ obtained in Proposition \ref{prop formula for HF on tree} satisfies
$$
\|f_\lambda\|_{\h_E} < \infty \ \ \mathrm{if\ and\ only\ if} \ \  \lambda >1.
$$
\end{lemma}

\begin{proof} By definition of $f_\lambda$, this function is constant on each subtree $T_i$, whose root is $x_i(1)$, and $x_{i+1}(2)$ is the unique neighbor of the root in $T_i$. By symmetry, the value of  $f_\lambda$ on the subtree $T'_i$, whose root is $x_i(2^i)$, is opposite to that on $T_i$. Moreover, by (\ref{eq f_lambda}),
$$
f_\lambda(x_i(1)) = \frac{\lambda^i -1}{\lambda^{i-2}(\lambda - 1)}.
$$
Then, the formula for the energy norm
$$
\|f_\lambda\|^2_{\h_E} = \frac{1}{2} \sum_{x\in V}\sum_{y\sim x} c_{xy} (f_\lambda(x) - f_\lambda(y))^2,
$$
gives nonzero contributions only for two infinite paths, $(x_0, x_1(1), ... , x_i(1), ...)$ and  $(x_0, x_1(2), ... , x_i(2^i), ...)$; furthermore,  these contributions are equal.
Thus,  $\|f_\lambda\|_{\h_E} < \infty$ if and only if the following series converges:
$$
\sum_{i =1}^\infty \lambda^{i-1}(f_\lambda(x_i(1)) - f_\lambda(x_{i-1}))^2 + \lambda^{i}(f_\lambda(x_i(1)) - f_\lambda(x_{i+1}))^2
$$
$$
= \sum_{i =1}^\infty \left( \left[\frac{\lambda^i -1}{\lambda^{i-2}(\lambda -1)}  - \frac{\lambda^{i-1} -1}{\lambda^{i-3}(\lambda -1)} \right]^2 \lambda^{i-1}  + \left[\frac{\lambda^i -1}{\lambda^{i-2}(\lambda -1)}  - \frac{\lambda^{i+1} -1}{\lambda^{i- 1}(\lambda -1)} \right]^2 \lambda^{i}  \right)
$$
$$
= \sum_{i =1}^\infty \left(\frac{1}{\lambda^{i-3}} + \frac{1}{\lambda^{i-2}} \right).
$$
The latter is obviously finite only for $\lambda >1$.
\end{proof}

\end{example}

\begin{example} [The Pascal graph] It is not difficult to realize the condition of Corollary  \ref{cor harm functions have inf energy} for a weighted Bratteli diagram $(B, c)$. For instance, suppose that $c_{xy} = 1$ for any edge $e = (xy)$, that is it defines the so called simple random walk on $B$. If the growth of the sequence $(|V_n|)$ is at most sublinear, then the series $\sum_{n =0}^\infty (\beta_n |V_n|)^{-1}$ diverges, and therefore any nonconstant harmonic function has infinite energy.

For example, this case is easily realized for the Pascal graph that we considered in Section \ref{sect HF on trees stationary BD}. When $c =1$ on the Pascal graph $B$, then Proposition \ref{prop HF on Pascal} gives us the explicitly defined harmonic  function $h$. It follows from Corollary \ref{cor harm functions have inf energy} that $\|h\|_{\h_E} = \infty$.
Moreover, we can {\em claim} that there is no harmonic functions of finite energy on the Pascal graph with the conductance function $c =1$.

On the other hand, there is a relatively simple formula for the energy norm of any harmonic function on the Pascal graph. We notice that the relation
$$
\| f\|^2_{\mathcal H_E} = \frac{1}{2}\sum_{x \in V} c(x) ((Pf^2)(x) - f^2(x)),
$$
used in Lemma \ref{lem for energy of harm fns}, can also be written (after some evident manipulation) as follows
$$
\| f\|^2_{\mathcal H_E} = \frac{1}{2}\sum_{x \in V} \sum_{y\sim x}c_{xy} (f^2(y) - f^2(x)).
$$
In the case when $c_{xy} = 1$, for all edges $e =(xy)$, we obtain that, for any nonconstant harmonic function $f$, the energy is
\be
\sum_{x \in V} \sum_{y\sim x} (f^2(y) - f^2(x)) = \infty.
\ee
A similar formula can be written for $c_{xy} = \lambda^n $ for any $e =(xy) \in E_n$. \end{example}

\begin{example}[Stationary Bratteli diagram]
In Proposition \ref{prop stat BD}, we described arbitrary harmonic function on a class of stationary Bratteli diagrams with $c_e \in \{\lambda^n : n \in \N_0\}$. We can find out under what condition a  harmonic function $f = (f_n)$ satisfying (\ref{eq HF stat BD}) has finite energy.

\begin{proposition} Suppose that  a stationary weighted Bratteli diagram $(B,c)$ satisfies conditions of Proposition \ref{prop stat BD} with $c_e = \lambda^n, e \in E_n$, and $\lambda > 1$.  Let $f = (f_n)$ be a harmonic function defined by (\ref{eq HF stat BD}). Then
$$
\|f\|_{\h_E} < \infty
$$
if and only if the vector $f_1(x)$ is constant.
\end{proposition}

\begin{proof}
We recall that, for any $x\in V_n$, we obtain from (\ref{eq HF stat BD}) the relation $f_n(x) = f_1(x) \dfrac{\lambda^n - 1}{\lambda^{n-1}(\lambda -1)}$. Then we  compute
\begin{eqnarray*}
  ||f||^2_{\h_E} &=& \frac{1}{2} \sum_{x\in V}\sum_{y\sim x} c_{xy} (f(x) - f(y))^2  \\
   &= & \frac{1}{2} \sum_{n =1}^{\infty}\sum_{x\in V_n} \lambda^{n-1} \left[ \sum_{\substack{y\in V_{n-1} \\ y\sim x }} (f_n(x) - f_{n-1}(y))^2 +
\lambda \sum_{\substack{z\in V_{n+1} \\ z \sim x }} (f_n(x) - f_{n+1}(z))^2 \right].
\end{eqnarray*}
For every summand in the above formula, we have
\begin{align}\label{eq energy stat BD1}
\begin{split}
 \lambda^{n-1}(f_n(x) - f_{n-1}(y))^2  & =   (f_1(x)(\lambda^n -1)  - f_1(y)(\lambda^n -\lambda))^2\frac{1}{\lambda^{n-1}(\lambda -1)^2}
\\
  & = ((f_1(x) - f_1(y))\lambda^n + f_1(y)\lambda -f_1(x))\frac{1}{\lambda^{n-1}(\lambda -1)^2},
\end{split}
\end{align}
and similarly
\be\label{eq energy stat BD2}
\lambda^{n}(f_n(x) - f_{n+1}(z))^2 = ((f_1(x) - f_1(z))\lambda^{n+1} - f_1(x)\lambda + f_1(z))\frac{1}{\lambda^{n}(\lambda -1)^2}.
\ee
Substituting (\ref{eq energy stat BD1}) and (\ref{eq energy stat BD2}) in the formula for the energy norm, we immediately deduce that $ ||f||^2_{\h_E} < \infty$ if and only if $f_1(x) = f_1(y)$ for any $x,y \in V_1$.
\end{proof}

\end{example}

\begin{remark} In \cite{Georgakopoulos2010}, it was proved that if for an electrical network $(G, c)$ the total conductance is finite, i.e., $\sum_{e \in E} c(e) < \infty$, then there is no nontrivial harmonic function of finite energy.  If the network $G$ is represented by a weighted Bratteli diagram $(B, c)$, we have
$$
\sum_{e \in E} c(e) = \frac{1}{2} \sum_x \sum_y c_{xy} = \frac{1}{2} \sum_{n=0}^\infty \sum_{x \in V_n} c(x).
$$
Then we deduce from \cite{Georgakopoulos2010} that if, in particular,  $\sum_{n=0}^\infty \beta_n |V_n| < \infty$, then there is no nonconstant harmonic function of finite energy on $(B,c)$.

Thus, we obtain the following qualitative observation: there two classes of Bratteli
diagrams when all harmonic functions have  infinite energy:
(i) the sequence $(\beta_n |V_n|)$ is either decreasing sufficiently fast, or (ii) it is
not growing too fast.
\end{remark}

\bibliographystyle{alpha}
\bibliography{bibliography}

\newcommand{\etalchar}[1]{$^{#1}$}
\def\ocirc#1{\ifmmode\setbox0=\hbox{$#1$}\dimen0=\ht0 \advance\dimen0
  by1pt\rlap{\hbox to\wd0{\hss\raise\dimen0
  \hbox{\hskip.2em$\scriptscriptstyle\circ$}\hss}}#1\else {\accent"17 #1}\fi}
\begin{thebibliography}{GHK{\etalchar{+}}15}

\bibitem[ALP99]{Ancona_Lyons_Peres1999}
Alano Ancona, Russell Lyons, and Yuval Peres.
\newblock Crossing estimates and convergence of {D}irichlet functions along
  random walk and diffusion paths.
\newblock {\em Ann. Probab.}, 27(2):970--989, 1999.

\bibitem[BH14]{Bezuglyi_Handelman2014}
Sergey Bezuglyi and David Handelman.
\newblock Measures on {C}antor sets: the good, the ugly, the bad.
\newblock {\em Trans. Amer. Math. Soc.}, 366(12):6247--6311, 2014.

\bibitem[BJKR02]{Bratteli_Jorgensen_Kim_Roush2002}
Ola Bratteli, Palle E.~T. Jorgensen, Ki~Hang Kim, and Fred Roush.
\newblock Computation of isomorphism invariants for stationary dimension
  groups.
\newblock {\em Ergodic Theory Dynam. Systems}, 22(1):99--127, 2002.

\bibitem[BK15]{Bezuglyi_Karpel2015}
S.~Bezuglyi and O.~Karpel.
\newblock Bratteli diagrams: structure, measures, dynamics.
\newblock {\em Preprint}, 2015.

\bibitem[BKMS10]{BKMS2010}
S.~Bezuglyi, J.~Kwiatkowski, K.~Medynets, and B.~Solomyak.
\newblock Invariant measures on stationary {B}ratteli diagrams.
\newblock {\em Ergodic Theory Dynam. Systems}, 30(4):973--1007, 2010.

\bibitem[Bot49]{Bott49}
Raoul Bott.
\newblock {\em Electrical network theory}.
\newblock ProQuest LLC, Ann Arbor, MI, 1949.
\newblock Thesis (Ph.D.)--Carnegie Mellon University.

\bibitem[Bra72]{Bratteli1972}
Ola Bratteli.
\newblock Inductive limits of finite dimensional {$C^{\ast} $}-algebras.
\newblock {\em Trans. Amer. Math. Soc.}, 171:195--234, 1972.

\bibitem[Car73]{Cartier1973}
Pierre Cartier.
\newblock G\'eom\'etrie et analyse sur les arbres.
\newblock In {\em S\'eminaire {B}ourbaki, 24\`eme ann\'ee (1971/1972), {E}xp.
  {N}o. 407}, pages 123--140. Lecture Notes in Math., Vol. 317. Springer,
  Berlin, 1973.

\bibitem[Chu10]{Chung2010}
Soon-Yeong Chung.
\newblock Identification of resistors in electrical networks.
\newblock {\em J. Korean Math. Soc.}, 47(6):1223--1238, 2010.

\bibitem[Die10]{Diekman2010}
Casey~O. Diekman.
\newblock {\em Modeling and {A}nalysis of {E}lectrical {N}etwork {A}ctivity in
  {N}euronal {S}ystems}.
\newblock ProQuest LLC, Ann Arbor, MI, 2010.
\newblock Thesis (Ph.D.)--University of Michigan.

\bibitem[DJ10]{Dutkay_Jorgensen2010}
Dorin~Ervin Dutkay and Palle E.~T. Jorgensen.
\newblock Spectral theory for discrete {L}aplacians.
\newblock {\em Complex Anal. Oper. Theory}, 4(1):1--38, 2010.

\bibitem[DJ11a]{Dutkay_Jorgensen2011-1}
Dorin~Ervin Dutkay and Palle E.~T. Jorgensen.
\newblock Affine fractals as boundaries and their harmonic analysis.
\newblock {\em Proc. Amer. Math. Soc.}, 139(9):3291--3305, 2011.

\bibitem[DJ11b]{Dutkay_Jorgensen2011}
Dorin~Ervin Dutkay and Palle E.~T. Jorgensen.
\newblock Spectral duality for unbounded operators.
\newblock {\em J. Operator Theory}, 65(2):325--353, 2011.

\bibitem[Du12]{Du2012}
Jie Du.
\newblock {\em On non-zero-sum stochastic game problems with stopping times}.
\newblock ProQuest LLC, Ann Arbor, MI, 2012.
\newblock Thesis (Ph.D.)--University of Southern California.

\bibitem[Dur10]{Durand_survey2010}
Fabien Durand.
\newblock Combinatorics on {B}ratteli diagrams and dynamical systems.
\newblock In {\em Combinatorics, automata and number theory}, volume 135 of
  {\em Encyclopedia Math. Appl.}, pages 324--372. Cambridge Univ. Press,
  Cambridge, 2010.

\bibitem[FKW90]{Furstenberg_Katznelson_Weiss1990}
Hillel Furstenberg, Yitzchak Katznelson, and Benjamin Weiss.
\newblock Ergodic theory and configurations in sets of positive density.
\newblock In {\em Mathematics of {R}amsey theory}, volume~5 of {\em Algorithms
  Combin.}, pages 184--198. Springer, Berlin, 1990.

\bibitem[FW03]{Furstenberg_Weiss2003}
Hillel Furstenberg and Benjamin Weiss.
\newblock Markov processes and {R}amsey theory for trees.
\newblock {\em Combin. Probab. Comput.}, 12(5-6):547--563, 2003.
\newblock Special issue on Ramsey theory.

\bibitem[Geo10]{Georgakopoulos2010}
Agelos Georgakopoulos.
\newblock Uniqueness of electrical currents in a network of finite total
  resistance.
\newblock {\em J. Lond. Math. Soc. (2)}, 82(1):256--272, 2010.

\bibitem[GHK{\etalchar{+}}15]{Georgakopoulus_Haeseler_Keller2015}
Agelos Georgakopoulos, Sebastian Haeseler, Matthias Keller, Daniel Lenz, and
  Rados{\l}aw~K. Wojciechowski.
\newblock Graphs of finite measure.
\newblock {\em J. Math. Pures Appl. (9)}, 103(5):1093--1131, 2015.

\bibitem[GHP14]{Grimmett_Holvord_Peres2014}
Geoffrey~R. Grimmett, Alexander~E. Holroyd, and Yuval Peres.
\newblock Extendable self-avoiding walks.
\newblock {\em Ann. Inst. Henri Poincar\'e D}, 1(1):61--75, 2014.

\bibitem[GPS95]{Giordano_Putnam_Skau1995}
Thierry Giordano, Ian~F. Putnam, and Christian~F. Skau.
\newblock Topological orbit equivalence and {$C\sp *$}-crossed products.
\newblock {\em J. Reine Angew. Math.}, 469:51--111, 1995.

\bibitem[HPS92]{Herman_Putnam_Skau1992}
Richard~H. Herman, Ian~F. Putnam, and Christian~F. Skau.
\newblock Ordered {B}ratteli diagrams, dimension groups and topological
  dynamics.
\newblock {\em Internat. J. Math.}, 3(6):827--864, 1992.

\bibitem[JP10]{Jorgensen_Pearse2010}
Palle E.~T. Jorgensen and Erin Peter~James Pearse.
\newblock A {H}ilbert space approach to effective resistance metric.
\newblock {\em Complex Anal. Oper. Theory}, 4(4):975--1013, 2010.

\bibitem[JP11]{Jorgensen_Pearse2011}
Palle E.~T. Jorgensen and Erin P.~J. Pearse.
\newblock Resistance boundaries of infinite networks.
\newblock In {\em Random walks, boundaries and spectra}, volume~64 of {\em
  Progr. Probab.}, pages 111--142. Birkh\"auser/Springer Basel AG, Basel, 2011.

\bibitem[JP13]{Jorgensen_Pearse2013}
Palle E.~T. Jorgensen and Erin P.~J. Pearse.
\newblock A discrete {G}auss-{G}reen identity for unbounded {L}aplace
  operators, and the transience of random walks.
\newblock {\em Israel J. Math.}, 196(1):113--160, 2013.

\bibitem[JP14]{Jorgensen_Pearse2014}
Palle E.~T. Jorgensen and Erin P.~J. Pearse.
\newblock Spectral comparisons between networks with different conductance
  functions.
\newblock {\em J. Operator Theory}, 72(1):71--86, 2014.

\bibitem[JT15]{JorgensenTian2015}
Palle Jorgensen and Feng Tian.
\newblock Frames and factorization of graph {L}aplacians.
\newblock {\em Opuscula Math.}, 35(3):293--332, 2015.

\bibitem[KLSW15]{Keller_Lenz_Schmidt_Wirth2015}
Matthias Keller, Daniel Lenz, Marcel Schmidt, and Melchior Wirth.
\newblock Diffusion determines the recurrent graph.
\newblock {\em Adv. Math.}, 269:364--398, 2015.

\bibitem[KLW13]{Keller_Lenz_Warzel2013}
Matthias Keller, Daniel Lenz, and Simone Warzel.
\newblock On the spectral theory of trees with finite cone type.
\newblock {\em Israel J. Math.}, 194(1):107--135, 2013.

\bibitem[MP84]{Mokobodzki_Pinchon1984}
G.~Mokobodzki and D.~Pinchon, editors.
\newblock {\em Th\'eorie du potentiel}, volume 1096 of {\em Lecture Notes in
  Mathematics}. Springer-Verlag, Berlin, 1984.

\bibitem[NW59]{Nash-Williams1959}
C.~St. J.~A. Nash-Williams.
\newblock Random walk and electric currents in networks.
\newblock {\em Proc. Cambridge Philos. Soc.}, 55:181--194, 1959.

\bibitem[Pet12]{Petit2012}
Camille Petit.
\newblock Harmonic functions on hyperbolic graphs.
\newblock {\em Proc. Amer. Math. Soc.}, 140(1):235--248, 2012.

\bibitem[Pow76]{Powers1976}
Robert~T. Powers.
\newblock Resistance inequalities for the isotropic {H}eisenberg ferromagnet.
\newblock {\em J. Mathematical Phys.}, 17(10):1910--1918, 1976.

\bibitem[PS12]{Peres_Sousi2012}
Yuval Peres and Perla Sousi.
\newblock Brownian motion with variable drift: 0-1 laws, hitting probabilities
  and {H}ausdorff dimension.
\newblock {\em Math. Proc. Cambridge Philos. Soc.}, 153(2):215--234, 2012.

\bibitem[QZ11]{QuianZhang2011}
Da-qian Qian and Xiao-dong Zhang.
\newblock Potential distribution on random electrical networks.
\newblock {\em Acta Math. Appl. Sin. Engl. Ser.}, 27(3):549--559, 2011.

\bibitem[Sok13]{Sokol2013}
Alexander Sokol.
\newblock An elementary proof that the first hitting time of an open set by a
  jump process is a stopping time.
\newblock In {\em S\'eminaire de {P}robabilit\'es {XLV}}, volume 2078 of {\em
  Lecture Notes in Math.}, pages 301--304. Springer, Cham, 2013.

\bibitem[SZ09a]{SmaleZhou2009-1}
Steve Smale and Ding-Xuan Zhou.
\newblock Geometry on probability spaces.
\newblock {\em Constr. Approx.}, 30(3):311--323, 2009.

\bibitem[SZ09b]{Smale_Zhou2009}
Steve Smale and Ding-Xuan Zhou.
\newblock Online learning with {M}arkov sampling.
\newblock {\em Anal. Appl. (Singap.)}, 7(1):87--113, 2009.

\bibitem[TOI{\etalchar{+}}67]{Tsuchiay1967}
T.~Tsuchiya, T.~Ohtsuki, Y.~Ishizaki, H.~Watanabe, Y.~Kajitani, and G.~Kishi.
\newblock Topological degrees of freedom of electrical networks.
\newblock In {\em Proc. {F}ifth {A}nnual {A}llerton {C}onf. on {C}ircuit and
  {S}ystem {T}heory ({M}onticello, {I}ll., 1967)}, pages 644--653. Univ. of
  Illinois, Urbana, Ill., 1967.

\bibitem[Woe00]{Woess2000}
Wolfgang Woess.
\newblock {\em Random walks on infinite graphs and groups}, volume 138 of {\em
  Cambridge Tracts in Mathematics}.
\newblock Cambridge University Press, Cambridge, 2000.

\bibitem[Woe09]{Woess2009}
Wolfgang Woess.
\newblock {\em Denumerable {M}arkov chains}.
\newblock EMS Textbooks in Mathematics. European Mathematical Society (EMS),
  Z\"urich, 2009.
\newblock Generating functions, boundary theory, random walks on trees.

\bibitem[Yam79]{Yamasaki1979}
Maretsugu Yamasaki.
\newblock Discrete potentials on an infinite network.
\newblock {\em Mem. Fac. Sci. Shimane Univ.}, 13:31--44, 1979.

\end{thebibliography}

\end{document}